\documentclass[11pt]{article}
\usepackage[all]{xy}

\usepackage{amsfonts}
\usepackage{amsthm}
\usepackage{enumerate}
\usepackage{graphicx}
\usepackage{mathrsfs}
\usepackage{bm}
\usepackage{cite}
\usepackage{amssymb,amsmath} 
\usepackage{makecell}
\usepackage{tikz}
\usepackage{pgf}
\usepackage{tikz}
\usetikzlibrary{patterns}
\usepackage{pgffor}
\usepackage{pgfcalendar}
\usepackage{pgfpages}
\usepackage{shuffle,yfonts}
\usepackage{mathtools}
\DeclareFontFamily{U}{shuffle}{}
\DeclareFontShape{U}{shuffle}{m}{n}{ <-8>shuffle7 <8->shuffle10}{}

\newcommand{\si}{\sigma}

\newcommand\Res{{\rm Res}}

\newcommand{\bfk}{{\boldsymbol{\sl{k}}}}

\newcommand{\bfx}{{\boldsymbol{\sl{x}}}}

\newcommand\bfsi{{\boldsymbol \sigma}}

 \allowdisplaybreaks
\usetikzlibrary{arrows,shapes,chains}
 \allowdisplaybreaks

\catcode`!=11
\let\!int\int \def\int{\displaystyle\!int}
\let\!lim\lim \def\lim{\displaystyle\!lim}
\let\!sum\sum \def\sum{\displaystyle\!sum}
\let\!sup\sup \def\sup{\displaystyle\!sup}
\let\!inf\inf \def\inf{\displaystyle\!inf}
\let\!cap\cap \def\cap{\displaystyle\!cap}
\let\!max\max \def\max{\displaystyle\!max}
\let\!min\min \def\min{\displaystyle\!min}
\let\!frac\frac \def\frac{\displaystyle\!frac}
\catcode`!=12

\let\oldsection\section
\renewcommand\section{\setcounter{equation}{0}\oldsection}

\allowdisplaybreaks

\DeclareMathOperator{\Li}{Li}

\def\N{\mathbb{N}}
\def\Z{\mathbb{Z}}
\def\Q{\mathbb{Q}}

\def\S{\widetilde{S}}

\def\ze{\zeta}

\theoremstyle{plain}
\newtheorem{thm}{Theorem}[section]
\newtheorem{lem}[thm]{Lemma}
\newtheorem{cor}[thm]{Corollary}
\newtheorem{con}[thm]{Conjecture}

\theoremstyle{definition}

\newtheorem{exa}[thm]{Example}

\begin{document}
\title{\bf Contour Integrations and Parity Results of Hurwitz-type Cyclotomic Euler Sums}
\author{
{Hongyuan Rui\thanks{Email: rhy626514@163.com}}\\[1mm]
\small School of Mathematics, Sichuan University,\\ \small Chengdu 610064, P.R. China\\
}

\date{}
\maketitle

\noindent{\bf Abstract.} In this paper, we investigate the parity of three class of Hurwitz-type cyclotomic Euler sums using the methods of contour integration and residue computation, and derive explicit parity formulas for linear, quadratic, and some higher-order cases. Based on their connection with cyclotomic multiple Hurwitz polylogarithm functions, we further obtain certain parity results for these functions. At the end of the paper, we propose two conjectures regarding the parity and symmetry of multiple Hurwitz polylogarithm functions of arbitrary depth.
\medskip

\medskip


\section{Introduction}

In 1735, Euler solved the famous \emph{Basel problem} and discovered the following well-known formula:
\begin{align}\label{Baselquestion}
\frac1{1^2}+\frac{1}{2^2}+\frac1{3^2}+\frac1{4^2}+\cdots=\frac{\pi^2}{6}.
\end{align}
Furthermore, he provided a more general formula:
\begin{align}\label{EulerBasel}
\frac1{1^{2m}}+\frac{1}{2^{2m}}+\frac1{3^{2m}}+\frac1{4^{2m}}+\cdots=\frac{(-1)^{m-1}B_{2m}(2\pi)^{2m}}{2(2m)!}\quad (m\in \N),
\end{align}
where $B_k$ denotes the \emph{Bernoulli numbers} defined by
\[\frac{x}{e^x-1}=\sum_{n=0}^{\infty}{\frac{B_n}{n!}x^n}.\]
The series on the left side of the above formula is now known as the value of the Riemann zeta function at positive even integers, where the \emph{Riemann zeta function} is defined as:
\begin{align}\label{defnRiemannzeta}
\zeta(s):=\sum_{n=1}^\infty \frac1{n^s}\quad (\Re(s)>1).
\end{align}
In a paper published in 1776, Euler \cite{Berndt1985,Euler1776} studied double series of the following form (now referred to as \emph{double zeta star values} or \emph{linear Euler sums}):
\begin{align}\label{defnlinearEulerSums}
S_{p,q}:=\sum_{n=1}^\infty \frac{H_n^{(p)}}{n^q}\quad (p,q\in \N\ \text{and}\ q>1),
\end{align}
where $H_n^{(p)}$ stands the {\emph{generalized harmonic number}} of order $p$ defined by
\[H_n^{(p)}:=\sum_{k=1}^n \frac{1}{k^p}.\]
In \cite{Euler1776}, Euler examines in detail relations
between the double series \eqref{defnlinearEulerSums} and the series of the form \eqref{EulerBasel}. Research on this series had already appeared in correspondence between Goldbach and Euler as early as around 1742-1743. Euler elaborated a method to show that the double series $S_{p,q}$ can be evaluated in terms of zeta values in the following cases: $p=1$, $p=q$, $(p,q)=(2,4),(4,2)$ and $p+q\leq 13$ odd. He even conjectured that all $S_{p,q}$ when $p+q$ is odd can be expressed in terms of zeta values, a result that was proven by Nielsen \cite{Nielsen1906} in 1906.

In 1998, Flajolet and Salvy developed a contour integral theory to systematically study Euler sums of the following general form:
\begin{align}
{S_{{p_1p_2\cdots p_k},q}} := \sum\limits_{n = 1}^\infty  {\frac{{H_n^{\left( {{p_1}} \right)}H_n^{\left( {{p_2}} \right)} \cdots H_n^{\left( {{p_k}} \right)}}}
{{{n^{q}}}}},
\end{align}
where $p_j\in \N$ and $q\geq 2$. The quantity $p_1+\cdots+p_k+q$ is called the ``weight" of the sum, and the quantity $k$ is called the
``degree". When $k\geq 2$, it is referred to as a \emph{nonlinear Euler sum}. One of the most important results in their paper was the proof of the parity theorem for nonlinear Euler sums, which can be stated as follows: \emph{A nonlinear Euler sum $S_{p_1\cdots p_k,q}$ reduces to a combination of sums of lower orders whenever the weight $p_1+\cdots+p_k+q$ and the order $k$ are of the same parity.} Moreover, Flajolet and Salvy highlighted the connection between these Euler sums and multiple zeta values. By utilizing the stuffle relations (see \cite{H2000}), they demonstrated that an Euler sum of weight $w$ and degree $k$ can be expressed as a $\Z$-coefficient linear combination of multiple zeta values with weight $w$ and depth not exceeding $k+1$.  The explicit formulas of Euler sums via MZVs were established by Xu and Wang in \cite{Xu-Wang2020}. The \emph{multiple zeta values} (MZVs) are defined by (\cite{H1992,DZ1994})
\begin{align}
\zeta(\bfk)\equiv \zeta(k_1,\ldots,k_r):=\sum_{0<n_1<\cdots<n_r} \frac{1}{n_1^{k_1}\cdots n_r^{k_r}},
\end{align}
where $k_1,\ldots,k_r$ are positive integers and $k_r\geq 2$ (i.e. admissible). Here $r$ and $k_1+\cdots+k_r$ are called the \emph{depth} and \emph{weight}, respectively. The concept of multiple zeta values was independently introduced in the early 1990s by both Hoffman \cite{H1992} and Zagier \cite{DZ1994}. Due to their profound connections with various mathematical and physical disciplines such as knot theory, algebraic geometry, and theoretical physics, the study of multiple zeta values has attracted considerable interest from numerous mathematicians and physicists. Over more than three decades of development, this field has yielded abundant research achievements. For a comprehensive survey of results obtained before 2016, readers may refer to Zhao's excellent monograph \cite{Z2016}. Similarly, the investigation of parity properties of multiple zeta values constitutes a significant research topic in this field. In 1996, Borwein and Girgensohn \cite{Borwein-Girgensohn-1996} proposed the following parity conjecture: When the weight $w$ and the depth $r$ are of the opposite parity, $\zeta(\bfk)$ is $\Q[\pi^2]$ -linear combination of multiple zeta values of depth at most $r-1$. Tsumura \cite{Tsu-2004} provided the first proof of this conjecture, which was subsequently followed by new proofs and further generalizations from various researchers, including: Brown \cite{Brown2021}, Ihara-Kaneko-Zagier \cite{IKZ2006}, Jarossay \cite{Jarossay2020}, Machide \cite{Machide2016} and Panzer \cite{Panzer2017}. Regrettably, none of the aforementioned proofs were able to provide a general explicit formula. However, Hirose \cite{Hirose2025} recently established an explicit formula for the parity of multiple zeta values by employing the theory of multitangent functions developed by Bouillot \cite{Bouillot2014}.

Indeed, the contour integral theory and methodology developed by Flajolet and Salvy can be further extended and generalized. Recently, Xu and Wang \cite{Xu-Wang2020,Xu-Wang2023} investigated the parity of two classes of Euler sums involving odd harmonic numbers (referred to as Euler $T$-sums and Euler $\tilde{S}$-sums) using the contour integral theory developed by Flajolet and Salvy. These sums are defined as follows:
\begin{align}
&T_{p_1,p_2,\ldots,p_k,q}^{\si_1,\si_2,\ldots,\si_k,\si}
    =\sum_{n=1}^\infty\si^{n-1}\frac{h_{n-1}^{(p_1)}(\si_1)h_{n-1}^{(p_2)}(\si_2)\cdots
        h_{n-1}^{(p_k)}(\si_k)}{(n-1/2)^q}\,,\label{Tsum.Unify}\\
&\S_{p_1,p_2,\ldots,p_k,q}^{\si_1,\si_2,\ldots,\si_k,\si}
    =\sum_{n=1}^\infty\si^{n-1}\frac{h_n^{(p_1)}(\si_1)h_n^{(p_2)}(\si_2)\cdots
        h_n^{(p_k)}(\si_k)}{n^q}\,,\label{Stsum.Unify}
\end{align}
where $(p_1,p_2,\ldots,p_k,q)\in\mathbb{N}^{k+1}$ and $(\si_1,\si_2,\ldots,\si_k,\si)\in\{\pm 1\}^{k+1}$ with $(q,\si)\neq (1,1)$. The $h_n^{(p)}(\si)$ denotes the (alternating) odd harmonic number defined by
\begin{align*}
h_n^{(p)}(\si):=\sum_{k=1}^n \frac{\si^k}{(k-1/2)^p}.
\end{align*}
In his recent review paper \cite{Z2024}, Zhao presented a wealth of research progress on sum formulas for multiple $t$-values, multiple $T$-values, and the related level 2 multiple zeta values.

Xu and Wang investigated the parity properties of Euler $T$-sums and Euler $\tilde{S}$-sums, establishing explicit parity formulas for multiple $t$-values and multiple $T$-values with depth$\leq 3$ (see \cite[Thms. 40 and 52]{Xu-Wang2023}). For $\bfk:=(k_1,\ldots,k_r)\in \N^r$ and $\bfsi:=(\si_1,\ldots,\si_r)\in \{\pm 1\}^r$ and $(k_r,\si_r)\neq (1,1)$, the (alternating) \emph{multiple $t$-values} (MtVs) and  the (alternating) \emph{multiple $T$-values} (MTVs) are defined by (\cite{H2019,Xu-Wang2023})
\begin{align}
&t(\bfk;\bfsi):=\sum_{0<n_1<\cdots<n_r\atop n_i\ \text{odd}} \frac{2^{k_1+\cdots+k_r}\si_1^{n_1}\cdots \si_r^{n_r}}{n_1^{k_1}\cdots n_r^{k_r}}=\sum_{0<n_1<\cdots<n_r} \frac{\si_1^{n_1}\cdots \si_r^{n_r}}{(n_1-1/2)^{k_1}\cdots (n_r-1/2)^{k_r}},\\
&T(\bfk;\bfsi):=2^r\sum\limits_{0<n_1<\cdots<n_r} \frac{\si_1^{n_1}\cdots \si_r^{n_r}}{(2n_1-1)^{k_1} (2n_2-2)^{k_2}\cdots(2n_r-r)^{k_r}}.
\end{align}
Very recently, the authors of this paper \cite{Rui-Xu2025} investigated the parity results of cyclotomic Euler sums using the method of contour integration. They established several explicit formulas for cases of degree $\leq 3$, from which they derived certain parity properties of cyclotomic multiple zeta values. The \emph{cyclotomic Euler sum} is defined by
\begin{align}\label{defn-CyESs}
S_{p_1,\ldots, p_k;q}(x_1,\ldots,x_k;x):=\sum_{n=1}^\infty \frac{\zeta_{n}(p_1;x_1)\zeta_{n}(p_2;x_2)\cdots \zeta_{n}(p_k;x_k)}{n^q}x^n,
\end{align}
where $p_1,\ldots,p_k,q\in \N$ and $x_1,\ldots,x_k,x$ are all roots of unity with $(q,x)\neq (1,1)$. Here $\zeta_n(p;x)$ stands the finite sum of polylogarithm function defined by
\begin{align}
\zeta_n(p;x):=\sum_{k=1}^n \frac{x^k}{k^p}\quad (p\in \N,\ x\in [-1,1]),
\end{align}
and the \emph{polylogarithm function} $\Li_p(x)$ is defined by
\begin{align}
\Li_{p}(x):=\lim_{n\rightarrow \infty}\zeta_n(p;x)= \sum_{n=1}^\infty \frac{x^n}{n^p}\quad (x\in[-1,1],\ (p,x)\neq (1,1),\ p\in \N).
\end{align}
For any $(k_1,\dotsc,k_r)\in\N^r$, the classical \emph{multiple polylogarithm function} with $r$-variables is defined by
\begin{align}\label{defn-mpolyf}
\Li_{k_1,\dotsc,k_r}(x_1,\dotsc,x_r):=\sum_{0<n_1<\cdots<n_r} \frac{x_1^{n_1}\dotsm x_r^{n_r}}{n_1^{k_1}\dotsm n_r^{k_r}}
\end{align}
which converges if $|x_j\cdots x_r|<1$ for all $j=1,\dotsc,r$. It can be analytically continued to a multi-valued meromorphic function on $\mathbb{C}^r$ (see \cite{Zhao2007d}). In particular, if $(k_1,\dotsc, k_r)\in\N^r$ and $x_1,\ldots,x_r$ are $N$th roots of unity, we call them \emph{cyclotomic multiple zeta values of level $N$} which converges if $(k_r,x_r)\ne (1,1)$ (see \cite{YuanZh2014a} and \cite[Ch. 15]{Z2016}). Similarly, for $\bfk=(k_1,\ldots,k_r)\in \N^r$ and $\bfx=(x_1,\ldots,x_r)$ (all $x_j$ are $N$-th roots of unity) with $(k_r,x_r)\neq (1,1)$, we define the \emph{cyclotomic multiple $t$-value of level $N$} $t(k_1,\ldots,k_r;x_1,\ldots,x_r)$ and \emph{cyclotomic multiple $T$-value of level $N$} by
\begin{align}
&t(\bfk;\bfx):=\sum_{0<n_1<\cdots<n_r} \frac{x_1^{n_1}\cdots x_r^{n_r}}{(n_1-1/2)^{k_1}\cdots (n_r-1/2)^{k_r}},\\
&T(\bfk;\bfx):=2^r\sum\limits_{0<n_1<\cdots<n_r} \frac{x_1^{n_1}\cdots x_r^{n_r}}{(2n_1-1)^{k_1} (2n_2-2)^{k_2}\cdots(2n_r-r)^{k_r}}.
\end{align}

In this paper, we define the following three classes of \emph{Hurwitz-type cyclotomic Euler sums}:
\begin{align}
&S^{(a)}_{p_1,p_2,\ldots,p_r;q}(x_1,x_2,\ldots,x_r;x):=\sum_{n=1}^\infty \frac{\zeta_{n}(p_1;x_1;a)\zeta_{n}(p_2;x_2;a)\cdots \zeta_{n}(p_r;x_r;a)}{(n+a)^q}x^n,\label{defnS-ES1}\\
&\tilde{S}^{(a)}_{p_1,p_2,\ldots,p_r;q}(x_1,x_2,\ldots,x_r;x):=\sum_{n=1}^\infty \frac{\zeta_{n}(p_1;x_1;a)\zeta_{n}(p_2;x_2;a)\cdots \zeta_{n}(p_r;x_r;a)}{n^q}x^n,\label{defnS-ES2}\\
&R^{(a)}_{p_1,p_2,\ldots,p_r;q}(x_1,x_2,\ldots,x_r;x):=\sum_{n=1}^\infty \frac{\zeta_{n}(p_1;x_1)\zeta_{n}(p_2;x_2)\cdots \zeta_{n}(p_r;x_r)}{(n+a)^q}x^n,\label{defnS-ES3}
\end{align}
where $p_1,\ldots,p_r,q\in \N$, $x_1,\ldots,x_r,x$ are all roots of unity with $(q,x)\neq (1,1)$ and $a\in \mathbb{C}\setminus \Z$. We naturally refer to the aforementioned three types of Euler sums as \emph{Hurwitz-type cyclotomic Euler $S$-sums, Euler $\tilde{S}$-sums}, and \emph{Euler $R$-sums}, respectively.
Here $\zeta_n(p;x;a)$ represents the finite sum of Hurwitz polylogarithm function, defined as:
\begin{align}
\zeta_n(p;x;a):=\sum_{k=1}^n \frac{x^k}{(k+a)^p}\quad (p\in \N,\ x\in [-1,1],\ a\notin\N^{-}:=\{-1,-2,-3,\ldots\}),
\end{align}
and the \emph{Hurwitz polylogarithm function} $\Li_p(x;a+1)\ (a\in \mathbb{C}\setminus \N^-)$ is defined by
\begin{align}
\Li_{p}(x;a+1):=\lim_{n\rightarrow \infty}\zeta_n(p;x;a)= \sum_{n=1}^\infty \frac{x^n}{(n+a)^p}\quad (x\in[-1,1],\ (p,x)\neq (1,1),\ p\in \N).
\end{align}
More general, for any $(k_1,\dotsc,k_r)\in\N^r$, the  \emph{multiple Hurwitz polylogarithm function} with $r$-variables is defined by
\begin{align}\label{defn-mpolyf}
\Li_{k_1,\dotsc,k_r}(x_1,\dotsc,x_r;a+1):=\sum_{0<n_1<\cdots<n_r} \frac{x_1^{n_1}\dotsm x_r^{n_r}}{(n_1+a)^{k_1}\dotsm (n_r+a)^{k_r}}
\end{align}
which converges if $|x_j\cdots x_r|<1$ for all $j=1,\dotsc,r$, and $a\neq -1,-2,-3,\ldots$. If all $x_j = 1$, then it reduces to the \emph{multiple Hurwitz zeta function} (\cite{AkiyamaIshikawa,MXY2021}). It should be emphasized that in paper \cite{XC2025} by the Xu and in paper \cite{Wang-Xu2025} by Xu and Wang, the contour integration method was employed to investigate the parity properties of the following three types of cyclotomic Euler sums:
\begin{align*}
	S_{p_1,p_2,\ldots,p_r;q}(x_1,x_2,\ldots,x_r;x)&\equiv S^{(-1/2)}_{p_1,p_2,\ldots,p_r;q}(x_1,x_2,\ldots,x_r;x),\\
	\tilde{S}_{p_1,p_2,\ldots,p_r;q}(x_1,x_2,\ldots,x_r;x)&\equiv\tilde{S}^{(-1/2)}_{p_1,p_2,\ldots,p_r;q}(x_1,x_2,\ldots,x_r;x),\\
	R_{p_1,p_2,\ldots,p_r;q}(x_1,x_2,\ldots,x_r;x)&\equiv R^{(1/2)}_{p_1,p_2,\ldots,p_r;q}(x_1,x_2,\ldots,x_r;x).
\end{align*} Correspondingly, they obtained parity results for certain level 2 variants of cyclotomic multiple zeta values, such as cyclotomic multiple $t$-values and cyclotomic multiple $T$-values.

In comparison, studies on Hurwitz-type Euler sums remain relatively scarce. Some relevant results can be found in \cite{AC2020,BBD2008,XC2019} and references therein. In \cite{BBD2008}, D.~Borwein, J.~M.~Borwein, and D.M. Bradley proved that
\begin{align}\label{for-para-harm-num-hur-ze}
\sum\limits_{n=1}^{\infty}\frac{H_{n}^{(2m+1)}}{n^{2}-a^{2}}
&=\frac{1}{2}\sum\limits_{n=1}^{\infty}\frac{1}{n^{2m+1}(n^{2}-a^{2})}
+\frac{1}{2a}\sum\limits_{k=0}^{m}\zeta(2k)\left(\zeta(2m-2k+2;a)-\zeta(2m-2k+2;-a)\right)\nonumber\\
&\quad+\frac{1}{4a}\pi \cot(\pi a)\Big(\zeta(2m+1;a)+\zeta(2m+1;-a)-2\zeta(2m+1)\Big),
\end{align}
where $\ze(s;a)$ is \emph{Hurwitz zeta function} defined by ($a\neq 0,-1,-2,-3,\ldots$)
\[\ze(s;a):=\sum_{n=0}^\infty \frac{1}{(n+a)^s}\quad (\Re(s)>1).\]
Alzer and Choi \cite{AC2020} discussed the analytic continuations and mingling connections of parametric linear Euler sums
$S_{p,q}^{++}(a,b),S_{p,q}^{+-}(a,b),S_{p,q}^{-+}(a,b),S_{p,q}^{--}(a,b)$, which are defined by
\begin{align}\label{Defn-PLESs}
\begin{aligned}
&S_{p,q}^{++}(a,b):=\sum_{n=1}^\infty \frac{H_n^{(p)}(a)}{(n+b)^q},\quad S_{p,q}^{+-}(a,b):=\sum_{n=1}^\infty \frac{H_n^{(p)}(a)}{(n+b)^q}(-1)^{n-1},\\
&S_{p,q}^{-+}(a,b):=\sum_{n=1}^\infty \frac{{\bar H}_n^{(p)}(a)}{(n+b)^q},\quad S_{p,q}^{--}(a,b):=\sum_{n=1}^\infty \frac{{\bar H}_n^{(p)}(a)}{(n+b)^q}(-1)^{n-1},
\end{aligned}
\end{align}
where $a,b\in \mathbb{C}\setminus\N^-\ (\N^-:=\{-1,-2,-3,-4,\ldots\}), p,q\in \mathbb{C}$ are adjusted so that the involved series can converge. Here $H_n^{(p)}(a)$ and ${\bar H}_n^{(p)}(a)$ are the \emph{parametric harmonic numbers} of order $p$ and the \emph{alternating parametric harmonic numbers} of order $p$, respectively, defined by
\begin{align}
&H_n^{\left( p\right)}\left( a \right): = \sum\limits_{j = 1}^n {\frac{1}
{{{{\left( {j + a} \right)}^p}}}}\quad\text{and}\quad{\bar H}_n^{\left( p\right)}\left( a \right): = \sum\limits_{j = 1}^n {\frac{(-1)^{j-1}}
{{{{\left( {j + a} \right)}^p}}}}\quad (n\in \N,\ p\in \mathbb{C},\ a\in \mathbb{C}\setminus\N^-).\label{Defn-APHNSs}
\end{align}

The primary objective of this paper is to investigate the parity properties of three classes of Hurwitz-type cyclotomic Euler sums defined in \eqref{defnS-ES1}, \eqref{defnS-ES2} and \eqref{defnS-ES3} using the contour integral method. We establish a general parity theorem and derive explicit formulas for both linear and quadratic cases. Furthermore, by specializing parameter values, we obtain certain parity results for cyclotomic multiple $t$-values and cyclotomic multiple $T$-values.

\section{Preliminary Knowledge}
In their systematic study of nonlinear Euler sums in 1998, Flajolet and Salvy \cite{Flajolet-Salvy} primarily employed contour integration by evaluating residues of integrals of the form $\oint_{(\infty)} r(s)\xi(s) ds=0$, where $\oint_{(\infty)}$ denotes integration along large circles, that is, the limit of integrals $\oint_{|s|=\rho}$, and $\xi(s)$ is referred to as a kernel function, defined as
\begin{align*}
\xi(s)=\frac{\pi\cot(\pi s)\psi^{(p_1-1)}(-s)\psi^{(p_2-1)}(-s)\cdots \psi^{(p_k-1)}(-s)}{(p_1-1)!(p_2-1)!\cdots(p_k-1)!}
\end{align*} and $r(s)$ is a basis function, defined as $r(s)=1/s^q\ (\forall p_j\in\N,\ q\in \N\setminus \{1\})$. Here $\psi(s)$ denotes the the \emph{digamma function} defined by
\begin{align}\label{defn-classical-psi-funtion}
\psi(s)=-\gamma-\frac1{s}+\sum_{k=1}^\infty \left(\frac1{k}-\frac{1}{s+k}\right),
\end{align}
where $s\in\mathbb{C}\setminus \N_0^-$ and $\N_0^-:=\N^-\cup\{0\}=\{0,-1,-2,-3,\ldots\}$, where $\gamma$ denotes the \emph{Euler-Mascheroni constant}. The kernel function $\xi(s)$ is defined to satisfy the following two conditions:
 1). $\xi(s)$ is meromorphic in the whole complex plane. 2). $\xi(s)$ satisfies $\xi(s)=o(s)$ over an infinite collection of circles $\left| s \right| = {\rho _k}$ with ${\rho _k} \to \infty $.
A crucial step in their approach involved the application of \emph{Cauchy's residue theorem} as follows:
\begin{lem}\emph{(cf.\ \cite[Lem. 2.1]{Flajolet-Salvy})}\label{lem-redisue-thm}
Let $\xi(s)$ be a kernel function and let $r(s)$ be a rational function which is $O(s^{-2})$ at infinity. Then
\begin{align}\label{residue-}
\sum\limits_{\alpha  \in O} {{\mathop{\rm Res}}{{\left( {r(s)\xi(s)},\alpha  \right)}}}  + \sum\limits_{\beta  \in S}  {{\mathop{\rm Res}}{{\left( {r(s)\xi(s)}, \beta  \right)}}}  = 0,
\end{align}
where $S$ is the set of poles of $r(s)$ and $O$ is the set of poles of $\xi(s)$ that are not poles $r(s)$. Here ${\mathop{\rm Re}\nolimits} s{\left( {r(s)},\alpha \right)} $ denotes the residue of $r(s)$ at $s= \alpha.$
\end{lem}

Recently, the authors of this paper \cite{Rui-Xu2025} introduced extended trigonometric functions $\Phi(s;x)$ and generalized digamma functions $\phi(s;x)$. By examining contour integrals of the kernel function $\xi(s):=\Phi(s;x)\phi^{(p_1-1)}(s;x_1)\cdots\phi^{(p_r-1)}(s;x_r)$ and rational function $r(s):=s^{-q}$, and through residue calculations, they established parity theorems and explicit formulas for cyclotomic Euler sums \eqref{defn-CyESs}. The \emph{generalized digamma function} $\phi(s;x)$ and \emph{extended trigonometric function} $\Phi(s;x)$ as follows:
\begin{align}
\phi(s;x):=\sum_{k=0}^\infty \frac{x^k}{k+s}\quad (s\notin\N_0^-:=\{0,-1,-2,-3,\ldots\}),
\end{align}
where $x$ is an arbitrary complex number with $|x|\leq 1$ and $x\neq 1$, and
\begin{align}
\Phi(s;x):=\phi(s;x)-\phi\Big(-s;x^{-1}\Big)-\frac1{s},
\end{align}
where $x$ is a root of unity. The function $\Phi(s;x)$ is referred to as an extended trigonometric function because the original $\pi \cot(\pi s)$ function admits an analogous representation in terms of the classical digamma function:
\begin{align*}
\pi \cot(\pi s)=-\frac1{s}+\psi(-s)-\psi(s).
\end{align*}
The function $\phi(s;x)$ is a special case of the \emph{classical Lerch zeta function}. In a recent study, Vicente and Holgado \cite{VH2025} have investigated a Lerch-type zeta function associated with recurrence sequences of arbitrary degree.
When computing residues for contour integrals, the crucial step involves determining the orders of all poles. In \cite{Rui-Xu2025}, the authors provided both Laurent series expansions and Maclaurin series expansions for the $\Phi(s;x)$ function and $\phi(s;x)$ function at integer points.
\begin{lem}(\cite[Eqs. (2.4) and (2.5)]{Rui-Xu2025})\label{lem-rui-xu-one} For $p\in \N$, if $|s+n|<1\ (n\geq 0)$, then
\begin{align}\label{Lexp-phi--n-diffp-1}
\frac{\phi^{(p-1)}(s;x)}{(p-1)!} (-1)^{p-1}&=x^n\sum_{k=0}^\infty \binom{k+p-1}{p-1} \left((-1)^k \Li_{k+p}(x)+(-1)^p\zeta_n\Big(k+p;x^{-1}\Big)\right)(s+n)^k \nonumber\\&\quad+\frac{x^n}{(s+n)^p}\qquad (|s+n|<1,\ n\geq 0)
\end{align}
and
\begin{align}\label{Lexp-phi-n-diffp-1}
\frac{\phi^{(p-1)}(s;x)}{(p-1)!} (-1)^{p-1}&=x^{-n}\sum_{k=0}^\infty \binom{k+p-1}{p-1} (-1)^k  \left( \Li_{k+p}(x)-\zeta_{n-1}\Big(k+p;x\Big)\right)(s-n)^k \nonumber\\&\qquad\qquad\qquad\qquad (|s-n|<1,\ n\geq 1).
\end{align}
\end{lem}
\begin{lem}(\cite[Eq. (2.6)]{Rui-Xu2025})\label{lem-rui-xu-two} For $n\in \Z$,
\begin{align}\label{LEPhi-function}
\Phi(s;x)=x^{-n} \left(\frac1{s-n}+\sum_{m=0}^\infty \Big((-1)^m\Li_{m+1}(x)-\Li_{m+1}\Big(x^{-1}\Big)\Big)(s-n)^m \right).
\end{align}
\end{lem}

In this paper, we primarily investigate the three classes of Hurwitz-type cyclotomic Euler sums in \eqref{defnS-ES1}, \eqref{defnS-ES2} and \eqref{defnS-ES3} by examining the following three types of contour integrals:
\begin{align*}
&\oint\limits_{\left( \infty  \right)} \frac{\Phi(s;x)\phi^{(p_1-1)}(s+a;x_1)\cdots\phi^{(p_r-1)}(s+a;x_r)}{(p_1-1)!\cdots (p_r-1)!(s+a)^q}(-1)^{p_1+\cdots+p_r-r} ds=0,\\
&\oint\limits_{\left( \infty  \right)} \frac{\Phi(s;x)\phi^{(p_1-1)}(s+a;x_1)\cdots\phi^{(p_r-1)}(s+a;x_r)}{(p_1-1)!\cdots (p_r-1)!s^q}(-1)^{p_1+\cdots+p_r-r} ds=0,\\
&\oint\limits_{\left( \infty  \right)} \frac{\Phi(s;x)\phi^{(p_1-1)}(s;x_1)\cdots\phi^{(p_r-1)}(s;x_r)}{(p_1-1)!\cdots (p_r-1)!(s+a)^q}(-1)^{p_1+\cdots+p_r-r} ds=0.
\end{align*}
Therefore, to facilitate the computation of residues for the aforementioned contour integrals, we need to derive the Maclaurin series expansions of $\phi^{(p-1)}(s+a;x)$ at integer points and $\Phi(s;x)$ at $(-n - a)$ (where $n\in \N_0:=\N\cup \{0\})$.
\begin{lem}\label{lem-extend-xuzhou-one} For $p\in \N$, if $|s+n|<1\ (n\geq 0)$, then
\begin{align}\label{Lexp-phi--n-diffp-1}
&\frac{\phi^{(p-1)}(s+a;x)}{(p-1)!} (-1)^{p-1}\nonumber\\
&=x^n\sum_{k=0}^\infty \binom{k+p-1}{p-1} \left((-1)^k \Li_{k+p}(x;a)x^{-1}+(-1)^p \zeta_n\Big(k+p;x^{-1};-a\Big)\right)(s+n)^k
\end{align}
and if $|s-n|<1\ (n\geq 1)$
\begin{align}\label{Lexp-phi-n-diffp-1}
&\frac{\phi^{(p-1)}(s+a;x)}{(p-1)!} (-1)^{p-1}\nonumber\\
&=x^{-n-1}\sum_{k=0}^\infty \binom{k+p-1}{p-1} (-1)^k  \left( \Li_{k+p}(x;a)-\zeta_{n}\Big(k+p;x;a-1\Big)\right)(s-n)^k.
\end{align}
\end{lem}
\begin{proof}
If $|s+n|<1\ (n\geq 0)$, it follows directly from the definition that
\begin{align*}
\phi(s+a;x)=x^n \sum_{m=0}^\infty \left((-1)^m \Li_{m+1}(x;a)x^{-1}-\zeta_n\Big(m+1;x^{-1};-a\Big) \right)(s+n)^m.
\end{align*}
Taking the $(p-1)$th derivative with respect to $s$ on both sides of the above equation yields formula \eqref{Lexp-phi--n-diffp-1}. Similarly, if $|s-n|<1\ (n\geq 1)$, by a direct calculation, we obtain
\begin{align*}
\phi(s+a;x)=x^{-n-1} \sum_{m=0}^\infty (-1)^m \left( \Li_{m+1}(x;a)-\zeta_n\Big(m+1;x;a-1\Big) \right)(s-n)^m.
\end{align*}
Taking the $(p-1)$th derivative with respect to $s$ on both sides of the above equation yields formula \eqref{Lexp-phi-n-diffp-1}.
\end{proof}

\begin{lem}\label{lem-extend-xuzhou-two} If $|s+n+a|<1\ (n\geq 0)$, then
\begin{align}
\Phi(s;x)=x^n \sum_{m=0}^\infty \left((-1)^m \Li_{m+1}(x;1-a)-x \Li_{m+1}\Big(x^{-1};a\Big) \right)(s+n+a)^m.
\end{align}
\end{lem}
\begin{proof}
The proof of this lemma is also based on an elementary calculation, which we leave to interested readers to attempt.
\end{proof}

Finally, we also need to provide the power series expansion of the $\phi^{(p-1)}(s;x)$ function for an arbitrary complex number $-a$ and $a$ is not a natural number.
\begin{lem}\label{lem-extend-xuzhou-three} If $|s+a|<1\ (n\geq 0)$, then
\begin{align}
\frac{\phi^{(p-1)}(s;x)}{(p-1)!}(-1)^{p-1}=\sum_{k=0}^\infty (-1)^k\binom{k+p-1}{p-1}\Li_{k+p}(x;-a)x^{-1}(s+a)^k.
\end{align}
\end{lem}
\begin{proof}
By an elementary calculation, we deduce
\begin{align*}
\phi(s;x)=\sum_{k=0}^\infty (-1)^k\Li_{k+1}(x;-a)x^{-1}(s+a)^k.
\end{align*}
Taking the $(p-1)$th derivative with respect to s on both sides of the above equation completes the proof of this lemma.
\end{proof}

\section{Parity Results}
In this section, we will employ the method of contour integration to derive explicit formulas for the parity relations of the linear and quadratic cases of the three types of Hurwitz-type cyclotomic Euler sums corresponding to equations \eqref{defnS-ES1}, \eqref{defnS-ES2} and \eqref{defnS-ES3}. Furthermore, we will present three theorems stating the parity relations for these three types of Hurwitz-type cyclotomic Euler sums of arbitrary order. Additionally, several illustrative examples are provided.

\subsection{Hurwitz-type Cyclotomic Linear Euler Sums}

First, in this subsection, we employ contour integration and residue computation to present the results and relevant examples for the linear case of these three types of Hurwitz-type cyclotomic Euler sums.

\begin{thm}\label{thm-linearCES} Let $x,y$ be roots of unity, $a\in \mathbb{C}\setminus \N$ and $p,q\geq 1$ with $(p,y), (q,xy)\neq (1,1)$. We have
\begin{align}
&x S_{p;q}^{(a-1)}\Big(y;(xy)^{-1}\Big)-(-1)^{p+q} S_{p;q}^{(-a)}\Big(y^{-1};xy\Big)\nonumber\\
&=x \Li_{p}(y;a)\Li_q\Big((xy)^{-1};a\Big)+(-1)^{q} y^{-1} \Li_p(y;a)\Li_q(xy;1-a)+(-1)^{p+q-1}\Li_{p+q}(x;1-a)\nonumber\\
&\quad+(-1)^q \sum_{m=0}^{p-1}\binom{p+q-m-2}{q-1}\left((-1)^m \Li_{m+1}(x;1-a)-x\Li_{m+1}\Big(x^{-1};a\Big)\right)\Li_{p+q-m-1}(xy)\nonumber\\
&\quad+(-1)^q \sum_{m=0}^{q-1}\binom{p+q-m-2}{p-1}\left((-1)^m x\Li_{m+1}\Big(x^{-1};a\Big)-\Li_{m+1}(x;1-a)\right)\Li_{p+q-m-1}(y).
\end{align}
\end{thm}
\begin{proof}
The proof of this theorem is based on residue calculations of the following contour integral:
\begin{align*}
\oint\limits_{\left( \infty  \right)} F^{(a)}_{p,q}(x,y;s)ds:= \oint\limits_{\left( \infty  \right)} \frac{\Phi(s;x)\phi^{(p-1)}(s+a;y)}{(p-1)!(s+a)^q} (-1)^{p-1}ds=0.
\end{align*}
The integrand $F^{(a)}_{p,q}(x,y;s)$ has the following poles throughout the complex plane: 1. All integers (simple poles); 2. $-a$ (pole of order $p+q)$ and 3. $-(n+a)$ (for positive integer $n$, poles of order $p$). Applying Lemma \ref{lem-rui-xu-one}-\ref{lem-extend-xuzhou-two}, by direct calculations, we deduce the following residues
\begin{align*}
&\Res\left(F^{(a)}_{p,q}(x,y;s),n\right)=\frac{x^{-n}y^{-n-1}}{(n+a)^q}\left(\Li_p(y;a)-\zeta_n(p;y;a-1)\right)\quad (n\geq 0),\\
&\Res\left(F^{(a)}_{p,q}(x,y;s),-n\right)=(-1)^q\frac{(xy)^n}{(n-a)^q}\left(\Li_p(y;a)y^{-1}+(-1)^p \zeta_n\Big(p;y^{-1};-a\Big)\right)\quad (n\geq 1),\\
&\Res\left(F^{(a)}_{p,q}(x,y;s),-n-a\right)=\frac1{(p-1)!} \lim_{s\rightarrow -n-a} \frac{d^{p-1}}{ds^{p-1}}\left((s+n+a)^pF^{(a)}_{p,q}(x,y;s)\right)\\
&=(-1)^q\sum_{m=0}^{p-1} \binom{p+q-m-2}{q-1} \left((-1)^m \Li_{m+1}(x;1-a)-x \Li_{m+1}\Big(x^{-1};a\Big) \right) \frac{(xy)^n}{n^{p+q-m-1}}\quad (n\geq 1)
\end{align*}
and
\begin{align*}
&\Res\left(F^{(a)}_{p,q}(x,y;s),-a\right)=\frac1{(p+q-1)!} \lim_{s\rightarrow -a} \frac{d^{p+q-1}}{ds^{p+q-1}}\left((s+a)^{p+q}F^{(a)}_{p,q}(x,y;s)\right)\\
&=(-1)^{p+q-1}\Li_{p+q}(x;1-a)-x \Li_{p+q}\Big(x^{-1};a\Big)\\
&\quad+\sum_{m+k=q-1,\atop m,k\geq 0} (-1)^k\binom{k+p-1}{p-1}\Li_{k+p}(y)\left((-1)^m \Li_{m+1}(x;1-a)-x\Li_{m+1}\Big(x^{-1};a\Big) \right).
\end{align*}
From Lemma \ref{lem-redisue-thm}, we know that
\begin{align*}
&\sum_{n=0}^\infty \Res\left(F^{(a)}_{p,q}(x,y;s),n\right)+\sum_{n=1}^\infty \Res\left(F^{(a)}_{p,q}(x,y;s),-n\right) \\
&\quad+\sum_{n=1}^\infty \Res\left(F^{(a)}_{p,q}(x,y;s),-n-a\right) +\Res\left(F^{(a)}_{p,q}(x,y;s),-a\right)=0.
\end{align*}
Finally, combining these four contributions yields the statement of Theorem \ref{thm-linearCES}.
\end{proof}

\begin{exa}Setting $(p,q)=(1,2)$ in Theorem \ref{thm-linearCES}, we have
	\begin{align*}
	&xS_{1;2}^{(a-1)}\Big(y;(xy)^{-1}\Big)+S_{1;2}^{(-a)}\Big(y^{-1};xy\Big)\nonumber\\
	&=x \Li_{1}(y;a)\Li_2\Big((xy)^{-1};a\Big)+y^{-1}\Li_1(y;a)\Li_2(xy;1-a)+\Li_{3}(x;1-a)\\
		&\quad+\left(\Li_{1}(x;1-a)-x\Li_{1}\Big(x^{-1};a\Big)\right)\Li_{2}(xy)\nonumber\\
		&\quad+\left(x\Li_{1}\Big(x^{-1};a\Big)-\Li_{1}(x;1-a)\right)\Li_{2}(y)-\left(x\Li_{2}\Big(x^{-1};a\Big)+\Li_{2}(x;1-a)\right)\Li_{1}(y).
	\end{align*}
	Setting $(p,q)=(2,1)$ in Theorem \ref{thm-linearCES}, we have
	\begin{align*}
	&x S_{2;1}^{(a-1)}\Big(y;(xy)^{-1}\Big)+S_{2;1}^{(-a)}\Big(y^{-1};xy\Big)\nonumber\\
	&=x \Li_{2}(y;a)\Li_1\Big((xy)^{-1};a\Big)-y^{-1} \Li_2(y;a)\Li_1(xy;1-a)+\Li_{3}(x;1-a)\\
	&\quad-\left(\Li_{1}(x;1-a)-x\Li_{1}\Big(x^{-1};a\Big)\right)\Li_{2}(xy)\nonumber\\
	&\quad+\left(\Li_{2}(x;1-a)+x\Li_{2}\Big(x^{-1};a\Big)\right)\Li_{1}(xy)-\left(x\Li_{1}\Big(x^{-1};a\Big)-\Li_{1}(x;1-a)\right)\Li_{2}(y).
	\end{align*}
\end{exa}

\begin{thm}\label{thm-linearCES-2} Let $x,y$ be roots of unity, $a\in \mathbb{C}\setminus \N$ and $p,q\geq 1$ with $(p,y), (q,xy)\neq (1,1)$. We have
\begin{align}
&y^{-1}\tilde{S}^{(a-1)}_{p;q}\Big(y;(xy)^{-1}\Big)-(-1)^{p+q} \tilde{S}^{(-a)}_{p;q}\Big(y^{-1};xy\Big)\nonumber\\
&=y^{-1}\Li_q\Big((xy)^{-1}\Big)\Li_p(y;a)+(-1)^q y^{-1}\Li_{q}(xy)\Li_p(y;a)+(-1)^q\binom{p+q-1}{p-1}y^{-1}\Li_{p+q}(y;a)\nonumber\\
&\quad+(-1)^q (xy)^{-1}\sum_{m=0}^{p-1} \binom{p+q-m-2}{q-1}\Li_{p+q-m-1}(xy;a)\nonumber\\&\qquad\qquad\qquad\qquad\qquad\times\left((-1)^m\Li_{m+1}(x;1-a)-x \Li_{m+1}\Big(x^{-1};a\Big)\right)\nonumber\\
&\quad+(-1)^q y^{-1} \sum_{m=0}^{q-1} \binom{p+q-m-2}{p-1}\Li_{p+q-m-1}(y;a)\left((-1)^m \Li_{m+1}\Big(x^{-1}\Big)-\Li_{m+1}(x)\right).
\end{align}
\end{thm}

\begin{proof}
In the context of this paper, the proof of this theorem is based on residue computations of the following contour integral:
\begin{align*}
\oint\limits_{\left( \infty  \right)} G^{(a)}_{p,q}(x,y;s)ds:= \oint\limits_{\left( \infty  \right)} \frac{\Phi(s;x)\phi^{(p-1)}(s+a;y)}{(p-1)!s^q} (-1)^{p-1}ds=0.
\end{align*}
Obviously, $n (n\in\N),\ 0,$ and $-(n + a)\ (n\in \N_0)$ are the simple poles, $(q+1)$th-order poles, and $p$th-order poles of the integrand $G^{(a)}_{p,q}(x,y;s)$, respectively. Using Lemma \ref{lem-rui-xu-one}-\ref{lem-extend-xuzhou-two}, the following residue values can be obtained through direct calculation:
\begin{align*}
&\Res\left(G^{(a)}_{p,q}(x,y;s),-n\right)=(-1)^q\frac{(xy)^n}{n^q}\left(\Li_p(y;a)y^{-1}+(-1)^p\zeta_n\Big(p;y^{-1};-a\Big) \right) \quad (n\geq 1),\\
&\Res\left(G^{(a)}_{p,q}(x,y;s),n\right)=\frac{x^{-n}y^{-n-1}}{n^q}\left(\Li_p(y;a)-\zeta_n(p;y;a-1)\right)\quad (n\geq 1),\\
&\Res\left(G^{(a)}_{p,q}(x,y;s),-n-a\right)=\frac1{(p-1)!} \lim_{s\rightarrow -n-a} \frac{d^{p-1}}{ds^{p-1}}\left((s+n+a)^pG^{(a)}_{p,q}(x,y;s)\right)\quad (n\geq 0)\\
&=(-1)^q\sum_{m=0}^{p-1} \binom{p+q-m-2}{q-1} \left((-1)^m \Li_{m+1}(x;1-a)-x \Li_{m+1}\Big(x^{-1};a\Big) \right) \frac{(xy)^n}{(n+a)^{p+q-m-1}}
\end{align*}
and
\begin{align*}
&\Res\left(G^{(a)}_{p,q}(x,y;s),0\right)=\frac1{q!} \lim_{s\rightarrow 0} \frac{d^{q}}{ds^{q}}\left(s^{q+1}G^{(a)}_{p,q}(x,y;s)\right)\\
&=(-1)^q\binom{p+q-1}{p-1}y^{-1}\Li_{p+q}(y;a)\\&\quad+\sum_{m+k=q-1,\atop m,k\geq 0} \binom{k+p-1}{p-1}(-1)^k \Li_{k+p}(y;a)y^{-1}\left((-1)^m\Li_{m+1}(x)-\Li_{m+1}\Big(x^{-1}\Big)\right).
\end{align*}
Applying Lemma \ref{lem-redisue-thm} in conjunction with the aforementioned four residue values suffices to prove the theorem.
\end{proof}

\begin{exa}Setting $(p,q)=(1,2)$ in Theorem \ref{thm-linearCES-2}, we have
\begin{align*}
	&y^{-1}\tilde{S}^{(a-1)}_{1;2}\Big(y;(xy)^{-1}\Big)+\tilde{S}^{(-a)}_{1;2}\Big(y^{-1};xy\Big)\nonumber\\
	&=y^{-1}\Li_2\Big((xy)^{-1}\Big)\Li_1(y;a)+y^{-1}\Li_{2}(xy)\Li_1(y;a)+y^{-1}\Li_{3}(y;a)\nonumber\\
	&\quad+(xy)^{-1}\Li_{2}(xy;a)\left(\Li_{1}(x;1-a)-x\Li_{1}\Big(x^{-1};a\Big)\right)\nonumber\\
	&\quad+y^{-1}\Li_{2}(y;a)\left(\Li_{1}\Big(x^{-1}\Big)-\Li_{1}(x)\right)-y^{-1}\Li_{1}(y;a)\left(\Li_{2}\Big(x^{-1}\Big)+\Li_{2}(x)\right).
\end{align*}
Setting $(p,q)=(2,1)$ in Theorem \ref{thm-linearCES-2}, we have
\begin{align*}
	&y^{-1}\tilde{S}^{(a-1)}_{2;1}\Big(y;(xy)^{-1}\Big)+\tilde{S}^{(-a)}_{2;1}\Big(y^{-1};xy\Big)\nonumber\\
	&=y^{-1}\Li_1\Big((xy)^{-1}\Big)\Li_2(y;a)-y^{-1}\Li_{1}(xy)\Li_2(y;a)-2y^{-1}\Li_{3}(y;a)\nonumber\\
	&\quad-(xy)^{-1}\Li_{2}(xy;a)\left(\Li_{1}(x;1-a)-x \Li_{1}\Big(x^{-1};a\Big)\right)\\
	&\quad+(xy)^{-1}\Li_{1}(xy;a)\left(\Li_{2}(x;1-a)+x\Li_{2}\Big(x^{-1};a\Big)\right)-y^{-1}\Li_{2}(y;a)\left(\Li_{1}\Big(x^{-1}\Big)-\Li_{1}(x)\right).
\end{align*}
\end{exa}

\begin{thm}\label{thm-linearCES-3} Let $x,y$ be roots of unity, $a\in \mathbb{C}\setminus \N$ and $p,q\geq 1$ with $(p,y), (q,xy)\neq (1,1)$. We have
\begin{align}
&(xy)^{-1}R^{(a+1)}_{p;q}\Big(y;(xy)^{-1}\Big)-(-1)^{p+q}R_{p;q}^{(-a)}\Big(y^{-1};xy\Big)\nonumber\\
&=\Li_p(y)\Li_q\Big((xy)^{-1};a+1\Big)+(-1)^q(xy)^{-1}\Li_p(y)\Li_q(xy;-a)\nonumber\\&\quad+(-1)^q\binom{p+q-1}{p}(xy)^{-1}\Li_{p+q}(xy;-a)\nonumber\\
&\quad+(-1)^q (xy)^{-1} \sum_{m=0}^{p-1}\binom{p+q-m-2}{q-1}\left((-1)^m\Li_{m+1}(x)-\Li_{m+1}\Big(x^{-1}\Big)\right)\Li_{p+q-m-1}(xy;-a)\nonumber\\
&\quad+(-1)^q y^{-1} \sum_{m=0}^{q-1}\binom{p+q-m-2}{p-1}\Li_{p+q-m-1}(y;-a)\nonumber\\&\qquad\qquad\qquad\qquad\times\left((-1)^m\Li_{m+1}\Big(x^{-1};a+1\Big)-x^{-1}\Li_{m+1}(x;-a)\right).
\end{align}
\end{thm}
\begin{proof}
Similar to the proofs of Theorems \ref{thm-linearCES} and \ref{thm-linearCES-2} above, the proof of this theorem requires consideration of this type of contour integral:
\begin{align*}
\oint\limits_{\left( \infty  \right)} H^{(a)}_{p,q}(x,y;s)ds:= \oint\limits_{\left( \infty  \right)} \frac{\Phi(s;x)\phi^{(p-1)}(s;y)}{(p-1)!(s+a)^q} (-1)^{p-1}ds=0.
\end{align*}
Obviously, all positive integers are simple poles, all non-positive integers are poles of order $p+1$, and $s=-a$ is a pole of order $q$. Applying Lemmas \ref{lem-rui-xu-one}, \ref{lem-rui-xu-two}, \ref{lem-extend-xuzhou-two} and \ref{lem-extend-xuzhou-three}, we obtain
\begin{align*}
&\Res\left(H^{(a)}_{p,q}(x,y;s),n\right)=\frac{(xy)^{-n}}{(n+a)^q}\left(\Li_p(y)-\zeta_{n-1}(p;y)\right)\quad (n\in \N),\\
&\Res\left(H^{(a)}_{p,q}(x,y;s),-n\right)=\frac1{p!}\lim_{s\rightarrow -n} \frac{d^p}{ds^p}\left\{(s+n)^{p+1}H^{(a)}_{p,q}(x,y;s)\right\}\quad (n\in\N_0)\\
&=(-1)^q\binom{p+q-1}{p}\frac{(xy)^n}{(n-a)^{p+q}}+(-1)^q\frac{(xy)^n}{(n-a)^q}\left(\Li_p(y)+(-1)^p\zeta_n\Big(p;y^{-1}\Big)\right)\\
&\quad+(-1)^q\sum_{m=0}^{p-1} \binom{p+q-m-2}{q-1}\left((-1)^m\Li_{m+1}(x)-\Li_{m+1}\Big(x^{-1}\Big)\right)\frac{(xy)^n}{(n-a)^{p+q-m-1}}
\end{align*}
and
\begin{align*}
&\Res\left(H^{(a)}_{p,q}(x,y;s),-a\right)\\
&=\sum_{m+k=q-1,\atop m,k\geq 0} (-1)^k\binom{k+p-1}{p-1}\Li_{k+p}(y;-a)y^{-1}\left((-1)^m\Li_{m+1}(x;1-a)-x \Li_{m+1}\Big(x^{-1};a\Big)\right).
\end{align*}
Noting the fact that
\begin{align*}
(-1)^m\Li_{m+1}(x;1-a)-x \Li_{m+1}\Big(x^{-1};a\Big)=(-1)^m\Li_{m+1}\Big(x^{-1};a+1\Big)-x^{-1}\Li_{m+1}(x;-a),
\end{align*}
and applying Lemma \ref{lem-redisue-thm}, we deduce the desired result with an elementary calculation.
\end{proof}
\begin{exa}Setting $(p,q)=(1,2)$ in Theorem \ref{thm-linearCES-3}, we have
\begin{align*}
	&(xy)^{-1}R^{(a+1)}_{1;2}\Big(y;(xy)^{-1}\Big)+R_{1;2}^{(-a)}\Big(y^{-1};xy\Big)\nonumber\\
	&=\Li_1(y)\Li_2\Big((xy)^{-1};a+1\Big)+(xy)^{-1}\Li_1(y)\Li_2(xy;-a)\nonumber\\
	&\quad+2(xy)^{-1}\Li_{3}(xy;-a)+(xy)^{-1}\left(\Li_{1}(x)-\Li_{1}\Big(x^{-1}\Big)\right)\Li_{2}(xy;-a)\\
	&\quad+y^{-1}\Li_{2}(y;-a)\left(\Li_{1}\Big(x^{-1};a+1\Big)-x^{-1}\Li_{1}(x;-a)\right)\nonumber\\
	&\quad-y^{-1}\Li_{1}(y;-a)\left(\Li_{2}\Big(x^{-1};a+1\Big)+x^{-1}\Li_{2}(x;-a)\right).
	\end{align*}
	Setting $(p,q)=(2,1)$ in Theorem \ref{thm-linearCES-3}, we have
	\begin{align*}
		&(xy)^{-1}R^{(a+1)}_{2;1}\Big(y;(xy)^{-1}\Big)+R_{2;1}^{(-a)}\Big(y^{-1};xy\Big)\nonumber\\
		&=\Li_2(y)\Li_1\Big((xy)^{-1};a+1\Big)-(xy)^{-1}\Li_2(y)\Li_1(xy;-a)-(xy)^{-1}\Li_{3}(xy;-a)\nonumber\\
		&\quad-(xy)^{-1} \left(\Li_{1}(x)-\Li_{1}\Big(x^{-1}\Big)\right)\Li_{2}(xy;-a)+(xy)^{-1} \left(\Li_{2}(x)+\Li_{2}\Big(x^{-1}\Big)\right)\Li_{1}(xy;-a)\nonumber\\
		&\quad-y^{-1}\Li_{2}(y;-a)\left(\Li_{1}\Big(x^{-1};a+1\Big)-x^{-1}\Li_{1}(x;-a)\right).
	\end{align*}
\end{exa}
\subsection{Hurwitz-type Cyclotomic Quadratic Euler Sums}
Next, in this subsection, we utilize contour integration and residue computation to present the results and specific examples for the quadratic case of these three types of Hurwitz-type cyclotomic Euler sums.
\begin{thm}\label{thm-quadratic-CES-one} Let $x,x_1,x_2$ be roots of unity, $a\in \mathbb{C}\setminus \Z$, and $p_1,p_2,q\in\N$ with $(p_1,x_1), (p_2,x_2) $ and $ (q,xx_1x_2)\neq (1,1)$. We have
\begin{align}
&x S_{p_1,p_2;q}^{(a-1)}\Big(x_1,x_2;(xx_1x_2)^{-1}\Big)+(-1)^{p_1+p_2+q} S_{p_1,p_2;q}^{(-a)}\Big(x_1^{-1},x_2^{-1};xx_1x_2\Big)\nonumber\\
&=x S_{p_1;p_2+q}^{(a-1)}\Big(x_1;(xx_1)^{-1}\Big)+x S_{p_2;p_1+q}^{(a-1)}\Big(x_2;(xx_2)^{-1}\Big)\nonumber\\
&\quad+x\Li_{p_1}(x_1;a)S_{p_2;q}^{(a-1)}\Big(x_2;(xx_1x_2)^{-1}\Big)+x\Li_{p_2}(x_2;a)S_{p_1;q}^{(a-1)}\Big(x_1;(xx_1x_2)^{-1}\Big)\nonumber\\
&\quad-(-1)^{p_2+q}x_1^{-1}\Li_{p_1}(x_1;a)S_{p_2;q}^{(-a)}\Big(x_2^{-1};xx_1x_2\Big)-(-1)^{p_1+q}x_2^{-1}\Li_{p_2}(x_2;a)S_{p_1;q}^{(-a)}\Big(x_1^{-1};xx_1x_2\Big)\nonumber\\
&\quad+(-1)^{p_1+p_2+q}\Li_{p_1+p_2+q}(x;1-a)-x \Li_{p_1}(x_1;a)\Li_{p_2+q}\Big((xx_1)^{-1};a\Big)\nonumber\\&\quad-x \Li_{p_2}(x_2;a)\Li_{p_1+q}\Big((xx_2)^{-1};a\Big)-x\Li_{p_1}(x_1;a)\Li_{p_2}(x_2;a)\Li_q\Big((xx_1x_2)^{-1};a\Big)\nonumber\\
&\quad-(-1)^q(x_1x_2)^{-1}\Li_{p_1}(x_1;a)\Li_{p_2}(x_2;a)\Li_q(xx_1x_2;1-a)\nonumber\\
&\quad-\sum_{m+k=p_1+q-1,\atop m,k\geq 0} (-1)^k \binom{k+p_2-1}{p_2-1}\Li_{k+p_2}(x_2)\left((-1)^m\Li_{m+1}(x;1-a)-x \Li_{m+1}\Big(x^{-1};a\Big)\right)\nonumber\\
&\quad-\sum_{m+k=p_2+q-1,\atop m,k\geq 0} (-1)^k \binom{k+p_1-1}{p_1-1}\Li_{k+p_1}(x_1)\left((-1)^m\Li_{m+1}(x;1-a)-x \Li_{m+1}\Big(x^{-1};a\Big)\right)\nonumber\\
&\quad-(-1)^q\sum_{m=0}^{p_1+p_2-1}\binom{p_1+p_2+q-m-2}{q-1}\left((-1)^m\Li_{m+1}(x;1-a)-x \Li_{m+1}\Big(x^{-1};a\Big)\right)\nonumber\\&\qquad\qquad\qquad\qquad\qquad\times \Li_{p_1+p_2+q-m-1}(xx_1x_2)\nonumber\\
&\quad-\sum_{m+k_1+k_2=q-1,\atop m,k_1,k_2\geq 0}(-1)^{k_1+k_2}\binom{k_1+p_1-1}{p_1-1}\binom{k_2+p_2-1}{p_2-1}\Li_{k_1+p_1}(x_1)\Li_{k_2+p_2}(x_2)\nonumber\\&\qquad\qquad\qquad\qquad\qquad\times \left((-1)^m\Li_{m+1}(x;1-a)-x \Li_{m+1}\Big(x^{-1};a\Big)\right)\nonumber\\
&\quad-(-1)^q \sum_{m+k\leq p_2-1,\atop m,k\geq 0} \binom{k+p_1-1}{p_1-1}\binom{p_2+q-m-k-2}{q-1} \nonumber\\&\qquad\qquad\qquad\qquad\qquad\times\left((-1)^m\Li_{m+1}(x;1-a)-x \Li_{m+1}\Big(x^{-1};a\Big)\right)\nonumber\\&\quad\times\left((-1)^k\Li_{k+p_1}(x_1)\Li_{p_2+q-m-k-1}(xx_1x_2)+(-1)^{p_1}S_{k+p_1;p_2+q-m-k-1}\Big(x_1^{-1};xx_1x_2\Big)\right)\nonumber\\
&\quad-(-1)^q \sum_{m+k\leq p_1-1,\atop m,k\geq 0} \binom{k+p_2-1}{p_2-1}\binom{p_1+q-m-k-2}{q-1} \nonumber\\&\qquad\qquad\qquad\qquad\qquad\times\left((-1)^m\Li_{m+1}(x;1-a)-x \Li_{m+1}\Big(x^{-1};a\Big)\right)\nonumber\\&\quad\times\left((-1)^k\Li_{k+p_2}(x_2)\Li_{p_1+q-m-k-1}(xx_1x_2)+(-1)^{p_2}S_{k+p_2;p_1+q-m-k-1}\Big(x_2^{-1};xx_1x_2\Big)\right).
\end{align}
\end{thm}
\begin{proof}
The proof of this theorem is based on residue calculations of the following contour integral:
\[\oint\limits_{\left( \infty  \right)}F^{(a)}_{p_1p_2,q}(x,x_1,x_2;s)ds:= \oint\limits_{\left( \infty  \right)}\frac{\Phi(s;x)\phi^{(p_1-1)}(s+a;x_1)\phi^{(p_2-1)}(s+a;x_2)}{(p_1-1)!(p_2-1)!(s+a)^q} (-1)^{p_1+p_2}ds=0.\]
It is evident that the integrand $F^{(a)}_{p_1p_2,q}(x,x_1,x_2;s)$ possesses the following poles in the complex plane: 1. All integer points are simple poles; 2. $s=-a$ is a pole of order $p_1+p_2+q$; 3. $s=-n-a$ (where $n$ is a positive integer) is a pole of order $p_1+p_2$. Applying Lemmas \ref{lem-rui-xu-two} and \ref{lem-extend-xuzhou-one}, we can compute the residues at simple poles located at integer points as follows:
\begin{align*}
&\Res\left(F^{(a)}_{p_1p_2,q}(\cdot;s),n\right)=\frac{x^{-n}(x_1x_2)^{-n-1}}{(n+a)^q}\left(\Li_{p_1}(x_1;a)-\zeta_n(p_1;x_1;a-1)\right)
\\&\qquad\qquad\qquad\qquad\qquad\qquad\qquad\times\left(\Li_{p_2}(x_2;a)-\zeta_n(p_2;x_2;a-1)\right)\quad (n\in \N_0),\\
&\Res\left(F^{(a)}_{p_1p_2,q}(\cdot;s),-n\right)=(-1)^q \frac{(xx_1x_2)^n}{(n-a)^q}\left(\Li_{p_1}(x_1;a)x_1^{-1}+(-1)^{p_1}\zeta_n\Big(p_1;x_1^{-1};-a\Big)\right)\\&\qquad\qquad\qquad\qquad\qquad\qquad\qquad
\times\left(\Li_{p_2}(x_2;a)x_2^{-1}+(-1)^{p_2}\zeta_n\Big(p_2;x_2^{-1};-a\Big)\right)\quad (n\in \N).
\end{align*}
Applying Lemmas \ref{lem-rui-xu-one} and \ref{lem-extend-xuzhou-two}, after extensive calculations, the residues at the poles located at $-a$ and $-n - a\ (n\in \N)$ can be obtained as follows:
\begin{align*}
&\Res\left(F^{(a)}_{p_1p_2,q}(\cdot;s),-a\right)\\&=\frac{1}{(p_1+p_2+q-1)!}\lim_{s\rightarrow -a}\frac{d^{p_1+p_2+q-1}}{ds^{p_1+p_2+q-1}}\left\{(s+a)^{p_1+p_2+q}F_{p_1p_2,q}^{(a)}(x,x_1,x_2;s)\right\}\\
&=(-1)^{p_1+p_2+q-1}\Li_{p_1+p_2+q}(x;1-a)-x\Li_{p_1+p_2+q}\Big(x^{-1};a\Big)\\
&\quad+\sum_{m+k=p_1+q-1,\atop m,k\geq 0} (-1)^k \binom{k+p_2-1}{p_2-1}\Li_{k+p_2}(x_2)\left((-1)^m\Li_{m+1}(x;1-a)-x \Li_{m+1}\Big(x^{-1};a\Big)\right)\\
&\quad+\sum_{m+k=p_2+q-1,\atop m,k\geq 0} (-1)^k \binom{k+p_1-1}{p_1-1}\Li_{k+p_1}(x_1)\left((-1)^m\Li_{m+1}(x;1-a)-x \Li_{m+1}\Big(x^{-1};a\Big)\right)\\
&\quad+\sum_{m+k_1+k_2=q-1,\atop m,k_1,k_2\geq 0}(-1)^{k_1+k_2}\binom{k_1+p_1-1}{p_1-1}\binom{k_2+p_2-1}{p_2-1}\Li_{k_1+p_1}(x_1)\Li_{k_2+p_2}(x_2)\nonumber\\&\qquad\qquad\qquad\qquad\qquad\times \left((-1)^m\Li_{m+1}(x;1-a)-x \Li_{m+1}\Big(x^{-1};a\Big)\right)
\end{align*}
and for $n\in\N$,
\begin{align*}
&\Res\left(F^{(a)}_{p_1p_2,q}(\cdot;s),-n-a\right)\\&=\frac{1}{(p_1+p_2-1)!}\lim_{s\rightarrow -n-a}\frac{d^{p_1+p_2-1}}{ds^{p_1+p_2-1}}\left\{(s+n+a)^{p_1+p_2}F_{p_1p_2,q}^{(a)}(x,x_1,x_2;s)\right\}\\
&=(-1)^q\sum_{m=0}^{p_1+p_2-1}\binom{p_1+p_2+q-m-2}{q-1}\left((-1)^m\Li_{m+1}(x;1-a)-x \Li_{m+1}\Big(x^{-1};a\Big)\right)\\&\qquad\qquad\qquad\qquad\qquad\times \frac{(xx_1x_2)^n}{n^{p_1+p_2+q-m-1}}\\
&\quad+(-1)^q \sum_{m+k\leq p_2-1,\atop m,k\geq 0} \binom{k+p_1-1}{p_1-1}\binom{p_2+q-m-k-2}{q-1}\frac{(xx_1x_2)^n}{n^{p_2+q-m-k-1}} \nonumber\\&\quad\times\left((-1)^m\Li_{m+1}(x;1-a)-x \Li_{m+1}\Big(x^{-1};a\Big)\right)\left((-1)^k\Li_{k+p_1}(x_1)+(-1)^{p_1}\zeta_n\Big(k+p_1;x_1^{-1}\Big)\right)\\
&\quad+(-1)^q \sum_{m+k\leq p_1-1,\atop m,k\geq 0} \binom{k+p_2-1}{p_2-1}\binom{p_1+q-m-k-2}{q-1}\frac{(xx_1x_2)^n}{n^{p_1+q-m-k-1}} \nonumber\\&\quad\times\left((-1)^m\Li_{m+1}(x;1-a)-x \Li_{m+1}\Big(x^{-1};a\Big)\right)\left((-1)^k\Li_{k+p_2}(x_2)+(-1)^{p_2}\zeta_n\Big(k+p_2;x_2^{-1}\Big)\right).
\end{align*}
By Lemma \ref{lem-redisue-thm}, we have
\begin{align*}
&\sum_{n=0}^\infty \Res\left(F^{(a)}_{p_1p_2,q}(\cdot;s),n\right)+\sum_{n=1}^\infty \Res\left(F^{(a)}_{p_1p_2,q}(\cdot;s),-n\right)\\&\quad+\sum_{n=1}^\infty\Res\left(F^{(a)}_{p_1p_2,q}(\cdot;s),-n-a\right)+\Res\left(F^{(a)}_{p_1p_2,q}(\cdot;s),-a\right)=0.
\end{align*}
Substituting the four residue results obtained above consequently proves Theorem \ref{thm-quadratic-CES-one}.
\end{proof}

\begin{exa}Setting $(p_1,p_2,q)=(1,1,2)$ in Theorem \ref{thm-quadratic-CES-one}, we have
\begin{align*}
	&xS_{1,1;2}^{(a-1)}\Big(x_1,x_2;(xx_1x_2)^{-1}\Big)+S_{1,1;2}^{(-a)}\Big(x_1^{-1},x_2^{-1};xx_1x_2\Big)\nonumber\\
	&=xS_{1;3}^{(a-1)}\Big(x_1;(xx_1)^{-1}\Big)+xS_{1;3}^{(a-1)}\Big(x_2;(xx_2)^{-1}\Big)\nonumber\\
	&\quad+x\Li_{1}(x_1;a)S_{1;2}^{(a-1)}\Big(x_2;(xx_1x_2)^{-1}\Big)+x\Li_{1}(x_2;a)S_{1;2}^{(a-1)}\Big(x_1;(xx_1x_2)^{-1}\Big)\nonumber\\
	&\quad+x_1^{-1}\Li_{1}(x_1;a)S_{1;2}^{(-a)}\Big(x_2^{-1};xx_1x_2\Big)+x_2^{-1}\Li_{1}(x_2;a)S_{1;2}^{(-a)}\Big(x_1^{-1};xx_1x_2\Big)\nonumber\\
	&\quad+\Li_{4}(x;1-a)-x\Li_{1}(x_1;a)\Li_{3}\Big((xx_1)^{-1};a\Big)\nonumber\\
	&\quad-x\Li_{1}(x_2;a)\Li_{3}\Big((xx_2)^{-1};a\Big)-x\Li_{1}(x_1;a)\Li_{1}(x_2;a)\Li_2\Big((xx_1x_2)^{-1};a\Big)\nonumber\\
	&\quad-(x_1x_2)^{-1}\Li_{1}(x_1;a)\Li_{1}(x_2;a)\Li_2(xx_1x_2;1-a)\nonumber\\
	&\quad-2\left(\Li_{1}(x;1-a)-x\Li_{1}\Big(x^{-1};a\Big)\right)\Li_{3}(xx_1x_2)+\left(\Li_{2}(x;1-a)+x\Li_{2}\Big(x^{-1};a\Big)\right)\Li_{2}(xx_1x_2)\nonumber\\
	&\quad-\left(\Li_{1}(x;1-a)-x\Li_{1}\Big(x^{-1};a\Big)\right)\left(\Li_{1}(x_1)\Li_{2}(xx_1x_2)-S_{1;2}\Big(x_1^{-1};xx_1x_2\Big)\right)\\
	&\quad-\left(\Li_{1}(x;1-a)-x\Li_{1}\Big(x^{-1};a\Big)\right)\left(\Li_{1}(x_2)\Li_{2}(xx_1x_2)-S_{1;2}\Big(x_2^{-1};xx_1x_2\Big)\right)\\
	&\quad-\Li_{2}(x_2)\left(\Li_{2}(x;1-a)+x\Li_{2}\Big(x^{-1};a\Big)\right)-\Li_{3}(x_2)\left(\Li_{1}(x;1-a)-x\Li_{1}\Big(x^{-1};a\Big)\right)\\
	&\quad-\Li_{1}(x_2)\left(\Li_{3}(x;1-a)-x\Li_{3}\Big(x^{-1};a\Big)\right)-\Li_{2}(x_1)\left(\Li_{2}(x;1-a)+x\Li_{2}\Big(x^{-1};a\Big)\right)\\
	&\quad-\Li_{3}(x_1)\left(\Li_{1}(x;1-a)-x\Li_{1}\Big(x^{-1};a\Big)\right)-\Li_{1}(x_1)\left(\Li_{3}(x;1-a)-x\Li_{3}\Big(x^{-1};a\Big)\right)\\
	&\quad+\Li_{1}(x_1)\Li_{1}(x_2)\left(\Li_{2}(x;1-a)+x \Li_{2}\Big(x^{-1};a\Big)\right)\\
	&\quad+\Li_{2}(x_1)\Li_{1}(x_2)\left(\Li_{1}(x;1-a)-x \Li_{1}\Big(x^{-1};a\Big)\right)\\
	&\quad+\Li_{1}(x_1)\Li_{2}(x_2)\left(\Li_{1}(x;1-a)-x \Li_{1}\Big(x^{-1};a\Big)\right).
\end{align*}
Setting $(p_1,p_2,q)=(1,2,2)$ in Theorem \ref{thm-quadratic-CES-one}, we have
\begin{align*}
	&xS_{1,2;2}^{(a-1)}\Big(x_1,x_2;(xx_1x_2)^{-1}\Big)-S_{1,2;2}^{(-a)}\Big(x_1^{-1},x_2^{-1};xx_1x_2\Big)\nonumber\\
	&=xS_{1;4}^{(a-1)}\Big(x_1;(xx_1)^{-1}\Big)+xS_{2;3}^{(a-1)}\Big(x_2;(xx_2)^{-1}\Big)\nonumber\\
	&\quad+x\Li_{1}(x_1;a)S_{2;2}^{(a-1)}\Big(x_2;(xx_1x_2)^{-1}\Big)+x\Li_{2}(x_2;a)S_{1;2}^{(a-1)}\Big(x_1;(xx_1x_2)^{-1}\Big)\nonumber\\
	&\quad-x_1^{-1}\Li_{1}(x_1;a)S_{2;2}^{(-a)}\Big(x_2^{-1};xx_1x_2\Big)+x_2^{-1}\Li_{2}(x_2;a)S_{1;2}^{(-a)}\Big(x_1^{-1};xx_1x_2\Big)\nonumber\\
	&\quad-\Li_{5}(x;1-a)-x\Li_{1}(x_1;a)\Li_{4}\Big((xx_1)^{-1};a\Big)\nonumber\\&\quad-x \Li_{2}(x_2;a)\Li_{3}\Big((xx_2)^{-1};a\Big)-x\Li_{1}(x_1;a)\Li_{2}(x_2;a)\Li_2\Big((xx_1x_2)^{-1};a\Big)\nonumber\\
	&\quad-(x_1x_2)^{-1}\Li_{1}(x_1;a)\Li_{2}(x_2;a)\Li_2(xx_1x_2;1-a)\nonumber\\
	&\quad-3\left(\Li_{1}(x;1-a)-x\Li_{1}\Big(x^{-1};a\Big)\right)\Li_{4}(xx_1x_2)\nonumber\\
	&\quad+2\left(\Li_{2}(x;1-a)+x\Li_{2}\Big(x^{-1};a\Big)\right)\Li_{3}(xx_1x_2)\nonumber\\
	&\quad-\left(\Li_{3}(x;1-a)-x\Li_{3}\Big(x^{-1};a\Big)\right)\Li_{2}(xx_1x_2)\nonumber\\
	&\quad-2\left(\Li_{1}(x;1-a)-x\Li_{1}\Big(x^{-1};a\Big)\right)\left(\Li_{1}(x_1)\Li_{3}(xx_1x_2)-S_{1;3}\Big(x_1^{-1};xx_1x_2\Big)\right)\nonumber\\
	&\quad+\left(\Li_{2}(x;1-a)+x\Li_{2}\Big(x^{-1};a\Big)\right)\left(\Li_{1}(x_1)\Li_{2}(xx_1x_2)-S_{1;2}\Big(x_1^{-1};xx_1x_2\Big)\right)\nonumber\\
	&\quad+\left(\Li_{1}(x;1-a)-x\Li_{1}\Big(x^{-1};a\Big)\right)\left(\Li_{2}(x_1)\Li_{2}(xx_1x_2)+S_{2;2}\Big(x_1^{-1};xx_1x_2\Big)\right)\nonumber\\
	&\quad-\left(\Li_{1}(x;1-a)-x\Li_{1}\Big(x^{-1};a\Big)\right)\left(\Li_{2}(x_2)\Li_{2}(xx_1x_2)+S_{2;2}\Big(x_2^{-1};xx_1x_2\Big)\right)\\
	&\quad-3\Li_{4}(x_2)\left(\Li_{1}(x;1-a)-x\Li_{1}\Big(x^{-1};a\Big)\right)-\Li_{2}(x_2)\left(\Li_{3}(x;1-a)-x\Li_{3}\Big(x^{-1};a\Big)\right)\\
	&\quad-2\Li_{3}(x_2)\left(\Li_{2}(x;1-a)+x\Li_{2}\Big(x^{-1};a\Big)\right)+\Li_{1}(x_1)\left(\Li_{4}(x;1-a)+x\Li_{4}\Big(x^{-1};a\Big)\right)\\
	&\quad+\Li_{2}(x_1)\left(\Li_{3}(x;1-a)-x\Li_{3}\Big(x^{-1};a\Big)\right)+\Li_{3}(x_1)\left(\Li_{2}(x;1-a)+x\Li_{2}\Big(x^{-1};a\Big)\right)\\
	&\quad+\Li_{4}(x_1)\left(\Li_{1}(x;1-a)-x\Li_{1}\Big(x^{-1};a\Big)\right)+\Li_{1}(x_1)\Li_{2}(x_2)\left(\Li_{2}(x;1-a)+x \Li_{2}\Big(x^{-1};a\Big)\right)\\
	&\quad+\Li_{2}(x_1)\Li_{2}(x_2)\left(\Li_{1}(x;1-a)-x\Li_{1}\Big(x^{-1};a\Big)\right)\\
	&\quad+2\Li_{1}(x_1)\Li_{3}(x_2)\left(\Li_{1}(x;1-a)-x\Li_{1}\Big(x^{-1};a\Big)\right).
\end{align*}
\end{exa}

\begin{thm}\label{thm-quadratic-CES-two} Let $x,x_1,x_2$ be roots of unity, $a\in \mathbb{C}\setminus \Z$, and $p_1,p_2,q\in\N$ with $(p_1,x_1), (p_2,x_2) $ and $ (q,xx_1x_2)\neq (1,1)$. We have
\begin{align}
	&(-1)^{p_1+p_2+q+1}\tilde{S}^{(-a)}_{p_1,p_2;q}\Big(x_1^{-1},x_2^{-1};xx_1x_2\Big)-(x_1x_2)^{-1}\tilde{S}^{(a-1)}_{p_1,p_2;q}\Big(x_1,x_2;(xx_1x_2)^{-1}\Big)\nonumber\\
	&=(x_1x_2)^{-1}\Li_{p_1}(x_1;a)\Li_{p_2}(x_2;a)\Li_q\Big((xx_1x_2)^{-1}\Big)-(x_1x_2)^{-1}\Li_{p_1}(x_1;a)\tilde{S}^{(a-1)}_{p_2;q}\Big(x_2;(xx_1x_2)^{-1}\Big)\nonumber\\
	&\quad-(x_1x_2)^{-1}\Li_{p_2}(x_2;a)\tilde{S}^{(a-1)}_{p_1;q}\Big(x_1;(xx_1x_2)^{-1}\Big)\nonumber\\
	&\quad+(-1)^q(x_1x_2)^{-1}\Li_{p_1}(x_1;a)\Li_{p_2}(x_2;a)\Li_q(xx_1x_2)+(-1)^{q+p_2}x_1^{-1}\Li_{p_1}(x_1;a)\tilde{S}^{(-a)}_{p_2;q}\Big(x_2^{-1};xx_1x_2\Big)\nonumber\\
	&\quad+(-1)^{q+p_1}x_2^{-1}\Li_{p_2}(x_2;a)\tilde{S}^{(-a)}_{p_1;q}\Big(x_1^{-1};xx_1x_2\Big)\nonumber\\
	&\quad+(-1)^q\sum_{k=0}^{p_1+p_2-1}\binom{p_1+p_2+q-k-2}{q-1}\left((-1)^k\Li_{k+1}(x;1-a)-x \Li_{k+1}\Big(x^{-1};a\Big)\right)\nonumber\\&\qquad\qquad\qquad\qquad\qquad\times(xx_1x_2)^{-1}\Li_{p_1+p_2+q-k-1}(xx_1x_2;a)\nonumber\\
	&\quad+(-1)^q \sum_{k_1+k_2\leq p_2-1,\atop k_1,k_2\geq 0} \binom{k_2+p_1-1}{p_1-1}\binom{p_2+q-k_1-k_2-2}{q-1}(xx_1x_2)^{-1} \nonumber\\
	&\quad\times\left((-1)^{k_1}\Li_{k_1+1}(x;1-a)-x \Li_{k_1+1}\Big(x^{-1};a\Big)\right)(-1)^{k_2}\Li_{k_2+p_1}(x_1)\Li_{p_2+q-k_1-k_2-1}(xx_1x_2;a)\nonumber\\
	&\quad+(-1)^q \sum_{k_1+k_2\leq p_2-1,\atop k_1,k_2\geq 0} \binom{k_2+p_1-1}{p_1-1}\binom{p_2+q-k_1-k_2-2}{q-1} \nonumber\\
	&\quad\times\left((-1)^{k_1}\Li_{k_1+1}(x;1-a)-x \Li_{k_1+1}\Big(x^{-1};a\Big)\right)(-1)^{p_1}R^{(a)}_{k_2+p_1;p_2+q-k_1-k_2-1}\Big(x_1^{-1};xx_1x_2\Big)\nonumber\\
	&\quad+(-1)^q \sum_{k_1+k_2\leq p_1-1,\atop k_1,k_2\geq 0} \binom{k_2+p_2-1}{p_2-1}\binom{p_1+q-k_1-k_2-2}{q-1}(xx_1x_2)^{-1} \nonumber\\
	&\quad\times\left((-1)^{k_1}\Li_{k_1+1}(x;1-a)-x \Li_{k_1+1}\Big(x^{-1};a\Big)\right)(-1)^{k_2}\Li_{k_2+p_2}(x_2)\Li_{p_1+q-k_1-k_2-1}(xx_1x_2;a)\nonumber\\
	&\quad+(-1)^q \sum_{k_1+k_2\leq p_1-1,\atop k_1,k_2\geq 0} \binom{k_2+p_2-1}{p_2-1}\binom{p_1+q-k_1-k_2-2}{q-1} \nonumber\\
	&\quad\times\left((-1)^{k_1}\Li_{k_1+1}(x;1-a)-x \Li_{k_1+1}\Big(x^{-1};a\Big)\right)(-1)^{p_2}R^{(a)}_{k_2+p_2;p_1+q-k_1-k_2-1}\Big(x_2^{-1};xx_1x_2\Big)\nonumber\\
	&\quad+\sum_{k_1+k_2=q,\atop k_1,k_2\geq 0}(x_1x_2)^{-1} (-1)^q \binom{k_1+p_1-1}{p_1-1}\binom{k_2+p_2-1}{p_2-1}\Li_{k_1+p_1}(x_1;a)\Li_{k_2+p_2}(x_2;a)\nonumber\\
	&\quad+\sum_{k_1+k_2+k_3=q-1,\atop k_1,k_2,k_3\geq 0}(-1)^{k_2+k_3}\binom{k_2+p_1-1}{p_1-1}\binom{k_3+p_2-1}{p_2-1}\Li_{k_2+p_1}(x_1;a)\Li_{k_3+p_2}(x_2;a)\nonumber\\
	&\quad\times(x_1x_2)^{-1}\left((-1)^{k_1}\Li_{{k_1}+1}(x)- \Li_{{k_1}+1}\Big(x^{-1}\Big)\right).
\end{align}
\end{thm}

\begin{proof}
The proof of this theorem is based on residue calculations of the following contour integral:
\[\oint\limits_{\left( \infty  \right)}G^{(a)}_{p_1p_2,q}(x,x_1,x_2;s)ds:= \oint\limits_{\left( \infty  \right)}\frac{\Phi(s;x)\phi^{(p_1-1)}(s+a;x_1)\phi^{(p_2-1)}(s+a;x_2)}{(p_1-1)!(p_2-1)!s^q} (-1)^{p_1+p_2}ds=0.\]
It is evident that the integrand $G^{(a)}_{p_1p_2,q}(x,x_1,x_2;s)$ possesses the following poles in the complex plane: 1. $n(n\in\mathbb{N})$ is a simple pole; 2. $s=0$ is a pole of order $q+1$; 3. $s=-n-a(n\in\mathbb{N}_0)$  is a pole of order $p_1+p_2$. Applying Lemmas \ref{lem-rui-xu-two} and \ref{lem-extend-xuzhou-one}, we can compute the residues at simple poles located at integer points as follows:
\begin{align*}
	&\Res\left(G^{(a)}_{p_1p_2,q}(\cdot;s),n\right)=\frac{x^{-n}(x_1x_2)^{-n-1}}{n^q}\left(\Li_{p_1}(x_1;a)-\zeta_n(p_1;x_1;a-1)\right)
	\\&\qquad\qquad\qquad\qquad\qquad\qquad\qquad\times\left(\Li_{p_2}(x_2;a)-\zeta_n(p_2;x_2;a-1)\right)\quad (n\in \N),\\
	&\Res\left(G^{(a)}_{p_1p_2,q}(\cdot;s),-n\right)=(-1)^q \frac{(xx_1x_2)^n}{n^q}\left(\Li_{p_1}(x_1;a)x_1^{-1}+(-1)^{p_1}\zeta_n\Big(p_1;x_1^{-1};-a\Big)\right)\\&\qquad\qquad\qquad\qquad\qquad\qquad\qquad
	\times\left(\Li_{p_2}(x_2;a)x_2^{-1}+(-1)^{p_2}\zeta_n\Big(p_2;x_2^{-1};-a\Big)\right)\quad (n\in \N).
\end{align*}
Applying Lemmas \ref{lem-rui-xu-one} and \ref{lem-extend-xuzhou-two}, after extensive calculations, the residues at the poles located at $0$ and $-n - a\ (n\in \N_0)$ can be obtained as follows:
\begin{align*}
	&\Res\left(G^{(a)}_{p_1p_2,q}(\cdot;s),0\right)\\
	&=\sum_{k_1+k_2=q,\atop k_1,k_2\geq 0}(x_1x_2)^{-1} (-1)^q \binom{k_1+p_1-1}{p_1-1}\binom{k_2+p_2-1}{p_2-1}\Li_{k_1+p_1}(x_1;a)\Li_{k_2+p_2}(x_2;a)\\
	&\quad+\sum_{k_1+k_2+k_3=q-1,\atop k_1,k_2,k_3\geq 0}(-1)^{k_2+k_3}\binom{k_2+p_1-1}{p_1-1}\binom{k_3+p_2-1}{p_2-1}\Li_{k_2+p_1}(x_1;a)\Li_{k_3+p_2}(x_2;a)\nonumber\\&\qquad\qquad\qquad\qquad\qquad\qquad\times(x_1x_2)^{-1} \left((-1)^{k_1}\Li_{{k_1}+1}(x)- \Li_{{k_1}+1}\Big(x^{-1}\Big)\right)
\end{align*}
and for $n\in\N_0$,
\begin{align*}
	&\Res\left(G^{(a)}_{p_1p_2,q}(\cdot;s),-n-a\right)\\
	&=(-1)^q\sum_{k=0}^{p_1+p_2-1}\binom{p_1+p_2+q-k-2}{q-1}\left((-1)^k\Li_{k+1}(x;1-a)-x \Li_{k+1}\Big(x^{-1};a\Big)\right)\\&\qquad\qquad\qquad\qquad\qquad\times \frac{(xx_1x_2)^n}{(n+a)^{p_1+p_2+q-k-1}}\\
	&\quad+(-1)^q \sum_{k_1+k_2\leq p_2-1,\atop k_1,k_2\geq 0} \binom{k_2+p_1-1}{p_1-1}\binom{p_2+q-k_1-k_2-2}{q-1}\frac{(xx_1x_2)^n}{(n+a)^{p_2+q-k_1-k_2-1}} \nonumber\\&\quad\times\left((-1)^{k_1}\Li_{k_1+1}(x;1-a)-x \Li_{k_1+1}\Big(x^{-1};a\Big)\right)\left((-1)^{k_2}\Li_{k_2+p_1}(x_1)+(-1)^{p_1}\zeta_n\Big(k_2+p_1;x_1^{-1}\Big)\right)\\
	&\quad+(-1)^q \sum_{k_1+k_2\leq p_1-1,\atop k_1,k_2\geq 0} \binom{k_2+p_2-1}{p_2-1}\binom{p_1+q-k_1-k_2-2}{q-1}\frac{(xx_1x_2)^n}{(n+a)^{p_1+q-k_1-k_2-1}} \nonumber\\&\quad\times\left((-1)^{k_1}\Li_{k_1+1}(x;1-a)-x \Li_{k_1+1}\Big(x^{-1};a\Big)\right)\left((-1)^{k_2}\Li_{k_2+p_2}(x_2)+(-1)^{p_2}\zeta_n\Big(k_2+p_2;x_2^{-1}\Big)\right).
\end{align*}
By Lemma \ref{lem-redisue-thm}, we have
\begin{align*}
	&\sum_{n=1}^\infty \Res\left(G^{(a)}_{p_1p_2,q}(\cdot;s),n\right)+\sum_{n=1}^\infty \Res\left(G^{(a)}_{p_1p_2,q}(\cdot;s),-n\right)\\&\quad+\sum_{n=0}^\infty\Res\left(G^{(a)}_{p_1p_2,q}(\cdot;s),-n-a\right)+\Res\left(G^{(a)}_{p_1p_2,q}(\cdot;s),0\right)=0.
\end{align*}
Substituting the four residue results obtained above consequently proves Theorem \ref{thm-quadratic-CES-two}.
\end{proof}

\begin{exa}Setting $(p_1,p_2,q)=(1,1,2)$ in Theorem \ref{thm-quadratic-CES-two}, we have
\begin{align*}
	&(x_1x_2)^{-1}\tilde{S}^{(a-1)}_{1,1;2}\Big(x_1,x_2;(xx_1x_2)^{-1}\Big)+\tilde{S}^{(-a)}_{1,1;2}\Big(x_1^{-1},x_2^{-1};xx_1x_2\Big)\\
	&=-(x_1x_2)^{-1}\Li_{1}(x_1;a)\Li_{1}(x_2;a)\Li_2\Big((xx_1x_2)^{-1}\Big)+(x_1x_2)^{-1}\Li_{1}(x_1;a)\tilde{S}^{(a-1)}_{1;2}\Big(x_2;(xx_1x_2)^{-1}\Big)\nonumber\\
	&\quad+(x_1x_2)^{-1}\Li_{1}(x_2;a)\tilde{S}^{(a-1)}_{1;2}\Big(x_1;(xx_1x_2)^{-1}\Big)-(x_1x_2)^{-1}\Li_{1}(x_1;a)\Li_{1}(x_2;a)\Li_2(xx_1x_2)\\
	&\quad+x_1^{-1}\Li_{1}(x_1;a)\tilde{S}^{(-a)}_{1;2}\Big(x_2^{-1};xx_1x_2\Big)+x_2^{-1}\Li_{1}(x_2;a)\tilde{S}^{(-a)}_{1;2}\Big(x_1^{-1};xx_1x_2\Big)\\
	&\quad-2(xx_1x_2)^{-1}\left(\Li_{1}(x;1-a)-x \Li_{1}\Big(x^{-1};a\Big)\right)\Li_{3}(xx_1x_2;a)\nonumber\\
	&\quad+(xx_1x_2)^{-1}\left(\Li_{2}(x;1-a)+x \Li_{2}\Big(x^{-1};a\Big)\right)\Li_{2}(xx_1x_2;a)\\
	&\quad-(xx_1x_2)^{-1}\left(\Li_{1}(x;1-a)-x \Li_{1}\Big(x^{-1};a\Big)\right)\Li_{1}(x_1)\Li_{2}(xx_1x_2;a)\nonumber\\
	&\quad+\left(\Li_{1}(x;1-a)-x\Li_{1}\Big(x^{-1};a\Big)\right)R^{(a)}_{1;2}\Big(x_1^{-1};xx_1x_2\Big)\nonumber\\
	&\quad-(xx_1x_2)^{-1}\left(\Li_{1}(x;1-a)-x \Li_{1}\Big(x^{-1};a\Big)\right)\Li_{1}(x_2)\Li_{2}(xx_1x_2;a)\nonumber\\
	&\quad+\left(\Li_{1}(x;1-a)-x\Li_{1}\Big(x^{-1};a\Big)\right)R^{(a)}_{1;2}\Big(x_2^{-1};xx_1x_2\Big)\nonumber\\
	&\quad-(x_1x_2)^{-1}\Li_{3}(x_1;a)\Li_{1}(x_2;a)-(x_1x_2)^{-1}\Li_{1}(x_1;a)\Li_{3}(x_2;a)\\
	&\quad-(x_1x_2)^{-1}\Li_{2}(x_1;a)\Li_{2}(x_2;a)\\
	&\quad+(x_1x_2)^{-1}\Li_{1}(x_1;a)\Li_{1}(x_2;a)\left(\Li_{2}(x)+ \Li_{2}\Big(x^{-1}\Big)\right)\\
	&\quad+(x_1x_2)^{-1}\Li_{2}(x_1;a)\Li_{1}(x_2;a)\left(\Li_{1}(x)-\Li_{1}\Big(x^{-1}\Big)\right)\\
	&\quad+(x_1x_2)^{-1}\Li_{1}(x_1;a)\Li_{2}(x_2;a)\left(\Li_{1}(x)-\Li_{1}\Big(x^{-1}\Big)\right).
\end{align*}
Setting $(p_1,p_2,q)=(1,2,2)$ in Theorem \ref{thm-quadratic-CES-two}, we have
\begin{align*}
	&(x_1x_2)^{-1}\tilde{S}^{(a-1)}_{1,2;2}\Big(x_1,x_2;(xx_1x_2)^{-1}\Big)-\tilde{S}^{(-a)}_{1,2;2}\Big(x_1^{-1},x_2^{-1};xx_1x_2\Big)\\
	&=-(x_1x_2)^{-1}\Li_{1}(x_1;a)\Li_{2}(x_2;a)\Li_2\Big((xx_1x_2)^{-1}\Big)+(x_1x_2)^{-1}\Li_{1}(x_1;a)\tilde{S}^{(a-1)}_{2;2}\Big(x_2;(xx_1x_2)^{-1}\Big)\nonumber\\
	&\quad+(x_1x_2)^{-1}\Li_{2}(x_2;a)\tilde{S}^{(a-1)}_{1;2}\Big(x_1;(xx_1x_2)^{-1}\Big)-(x_1x_2)^{-1}\Li_{1}(x_1;a)\Li_{2}(x_2;a)\Li_2(xx_1x_2)\nonumber\\
	&\quad-x_1^{-1}\Li_{1}(x_1;a)\tilde{S}^{(-a)}_{2;2}\Big(x_2^{-1};xx_1x_2\Big)+x_2^{-1}\Li_{2}(x_2;a)\tilde{S}^{(-a)}_{1;2}\Big(x_1^{-1};xx_1x_2\Big)\nonumber\\
	&\quad-3\left(\Li_{1}(x;1-a)-x\Li_{1}\Big(x^{-1};a\Big)\right)(xx_1x_2)^{-1}\Li_{4}(xx_1x_2;a)\nonumber\\
	&\quad+2\left(\Li_{2}(x;1-a)+x\Li_{2}\Big(x^{-1};a\Big)\right)(xx_1x_2)^{-1}\Li_{3}(xx_1x_2;a)\nonumber\\
	&\quad-\left(\Li_{3}(x;1-a)-x\Li_{3}\Big(x^{-1};a\Big)\right)(xx_1x_2)^{-1}\Li_2(xx_1x_2;a)\nonumber\\
	&\quad-2\left(\Li_{1}(x;1-a)-x\Li_{1}\Big(x^{-1};a\Big)\right)\left((xx_1x_2)^{-1}\Li_{1}(x_1)\Li_{3}(xx_1x_2;a)-R^{(a)}_{1;3}\Big(x_1^{-1};xx_1x_2\Big)\right)\nonumber\\
	&\quad+\left(\Li_{2}(x;1-a)+x\Li_{2}\Big(x^{-1};a\Big)\right)\left((xx_1x_2)^{-1}\Li_{1}(x_1)\Li_{2}(xx_1x_2;a)-R^{(a)}_{1;2}\Big(x_1^{-1};xx_1x_2\Big)\right)\nonumber\\
	&\quad+\left(\Li_{1}(x;1-a)-x\Li_{1}\Big(x^{-1};a\Big)\right)\left((xx_1x_2)^{-1}\Li_{2}(x_1)\Li_{2}(xx_1x_2;a)+R^{(a)}_{2;2}\Big(x_1^{-1};xx_1x_2\Big)\right)\nonumber\\
	&\quad-\left(\Li_{1}(x;1-a)-x\Li_{1}\Big(x^{-1};a\Big)\right)\left((xx_1x_2)^{-1}\Li_{2}(x_2)\Li_{2}(xx_1x_2;a)+R^{(a)}_{2;2}\Big(x_2^{-1};xx_1x_2\Big)\right)\nonumber\\
	&\quad-(x_1x_2)^{-1}\Li_{3}(x_1;a)\Li_{2}(x_2;a)-3(x_1x_2)^{-1}\Li_{1}(x_1;a)\Li_{4}(x_2;a)\\
	&\quad-2(x_1x_2)^{-1}\Li_{2}(x_1;a)\Li_{3}(x_2;a)\\
	&\quad+(x_1x_2)^{-1}\Li_{1}(x_1;a)\Li_{2}(x_2;a)\left(\Li_{2}(x)+ \Li_{2}\Big(x^{-1}\Big)\right)\\
	&\quad+(x_1x_2)^{-1}\Li_{2}(x_1;a)\Li_{2}(x_2;a)\left(\Li_{1}(x)-\Li_{1}\Big(x^{-1}\Big)\right)\\
	&\quad+2(x_1x_2)^{-1}\Li_{1}(x_1;a)\Li_{3}(x_2;a)\left(\Li_{1}(x)-\Li_{1}\Big(x^{-1}\Big)\right).
\end{align*}
\end{exa}

\begin{thm}\label{thm-quadratic-CES-three} Let $x,x_1,x_2$ be roots of unity, $a\in \mathbb{C}\setminus \Z$, and $p_1,p_2,q\in\N$ with $(p_1,x_1), (p_2,x_2) $ and $ (q,xx_1x_2)\neq (1,1)$. We have
	\begin{align}
	&(-1)^{p_1+p_2+q+1}R_{p_1,p_2;q}^{(-a)}\Big(x_1^{-1},x_2^{-1};xx_1x_2\Big)-(xx_1x_2)^{-1}R^{(a+1)}_{p_1,p_2;q}\Big(x_1,x_2;(xx_1x_2)^{-1}\Big)\nonumber\\
	&=\Li_{p_1}(x_1)\Li_{p_2}(x_2)\Li_q\Big((xx_1x_2)^{-1};a+1\Big)-(xx_1x_2)^{-1}\Li_{p_1}(x_1)R^{(a+1)}_{p_2;q}\Big(x_2;(xx_1x_2)^{-1}\Big)\nonumber\\
	&\quad-(xx_1x_2)^{-1}\Li_{p_2}(x_2)R^{(a+1)}_{p_1;q}\Big(x_1;(xx_1x_2)^{-1}\Big)\nonumber\\
	&\quad+(-1)^q(xx_1x_2)^{-1}\binom{q+p_1+p_2-1}{q-1}\Li_{q+p_1+p_2}(xx_1x_2;-a)\nonumber\\
	&\quad+(-1)^q\sum_{k=0}^{p_{1}}\binom{k+p_2-1}{p_{2}-1}\binom{q+p_{1}-k-1}{q-1}\nonumber\\
	&\quad\quad\times\left((-1)^k(xx_1x_2)^{-1}\Li_{k+1}(x_{2})\Li_{p_1+q-k}(xx_1x_{2};-a)+(-1)^{p_{2}}R_{k+p_2;p_1+q-k}^{(-a)}\Big(x_2^{-1};xx_1x_2\Big)\right)\nonumber\\
	&\quad+(-1)^q\sum_{k=0}^{p_{2}}\binom{k+p_1-1}{p_{1}-1}\binom{q+p_{2}-k-1}{q-1}\nonumber\\
	&\quad\quad\times\left((-1)^k(xx_1x_2)^{-1}\Li_{k+1}(x_{1})\Li_{p_2+q-k}(xx_1x_{2};-a)+(-1)^{p_{1}}R_{k+p_1;p_2+q-k}^{(-a)}\Big(x_1^{-1};xx_1x_2\Big)\right)\nonumber\\
	&\quad+(-1)^q(xx_1x_2)^{-1}\sum_{k=0}^{p_1+p_2-1}\binom{q+p_1+p_2-k-2}{q-1}\left((-1)^k\Li_{k+1}(x)-\Li_{k+1}\Big(x^{-1}\Big)\right)\nonumber\\
	&\quad\quad\times\Li_{q+p_1+p_2-k-1}(xx_1x_2;-a)\nonumber\\
	&\quad+(-1)^q(xx_1x_2)^{-1}\Li_{p_1}(x_1)\Li_{p_2}(x_2)\Li_q(xx_1x_2;-a)+(-1)^{q+p_2}\Li_{p_1}(x_1)R_{p_2;q}^{(-a)}\Big(x_2^{-1};xx_1x_2\Big)\nonumber\\
	&\quad+(-1)^{q+p_1}\Li_{p_2}(x_2)R_{p_1;q}^{(-a)}\Big(x_1^{-1};xx_1x_2\Big)\nonumber\\
	&\quad+(-1)^q\sum_{0\le k_1+k_2\le p_{1}-1}\binom{k_2+p_{2}-1}{p_{2}-1}\binom{q+p_{1}-k_1-k_2-2}{q-1}\nonumber\\
	&\quad\quad\times\left((-1)^{k_1}\Li_{k_1+1}(x)-\Li_{k_1+1}\Big(x^{-1}\Big)\right)\begin{Bmatrix}
		(-1)^{k_2}(xx_1x_2)^{-1}\Li_{k_2+p_2}(x_{2})\\
		\times\Li_{p_{1}+q-k_1-k_2-1}(xx_1x_{2};-a)\\+(-1)^{p_{2}}R^{(-a)}_{k_2+p_2;p_{1}+q-k_1-k_2-1}\Big(x_{2}^{-1};xx_1x_2\Big)
	\end{Bmatrix}\nonumber\\
	&\quad+(-1)^q\sum_{0\le k_1+k_2\le p_{2}-1}\binom{k_2+p_{1}-1}{p_{1}-1}\binom{q+p_{2}-k_1-k_2-2}{q-1}\nonumber\\
	&\quad\quad\times\left((-1)^{k_1}\Li_{k_1+1}(x)-\Li_{k_1+1}\Big(x^{-1}\Big)\right)\begin{Bmatrix}
		(-1)^{k_2}(xx_1x_2)^{-1}\Li_{k_2+p_1}(x_{1})\\
		\times\Li_{p_{2}+q-k_1-k_2-1}(xx_1x_{2};-a)\\+(-1)^{p_{1}}R^{(-a)}_{k_2+p_1;p_{2}+q-k_1-k_2-1}\Big(x_{1}^{-1};xx_1x_2\Big)
	\end{Bmatrix}\nonumber\\
	&\quad+\sum_{k_1+k_2+k_3=q-1,\atop k_1,k_2,k_3\geq 0}\left((-1)^{k_1}\Li_{{k_1}+1}(x;1-a)-x \Li_{{k_1}+1}\Big(x^{-1};a\Big) \right)\nonumber\\
	&\quad\quad\times(-1)^{k_2+k_3}\binom{k_2+p_1-1}{p_1-1}\binom{k_3+p_2-1}{p_2-1}\Li_{k_2+p_1}(x_1;-a)\Li_{k_3+p_2}(x_2;-a)(x_1x_2)^{-1}.
	\end{align}
\end{thm}

\begin{proof}
The proof of this theorem is based on residue calculations of the following contour integral:
\[\oint\limits_{\left( \infty  \right)}H^{(a)}_{p_1p_2,q}(x,x_1,x_2;s)ds:= \oint\limits_{\left( \infty  \right)}\frac{\Phi(s;x)\phi^{(p_1-1)}(s;x_1)\phi^{(p_2-1)}(s;x_2)}{(p_1-1)!(p_2-1)!(s+a)^q} (-1)^{p_1+p_2}ds=0.\]
It is evident that the integrand $H^{(a)}_{p_1p_2,q}(x,x_1,x_2;s)$ possesses the following poles in the complex plane: 1. all postive integers are simple poles; 2. all non-postive integers are poles of order $p_1+p_2+1$; 3. $s=-a$ is a pole of order $q$. Applying Lemmas \ref{lem-rui-xu-two} and \ref{lem-extend-xuzhou-one}, we can compute the residues at simple poles located at integer points as follows:
\begin{align*}
	&\Res\left(H^{(a)}_{p_1p_2,q}(x,x_1,x_2;s),n\right)\\
	&=\frac{(xx_1x_2)^{-n}}{(n+a)^q}\left(\Li_{p_1}(x_1)-\zeta_{n-1}(p_1;x_1)\right)\left(\Li_{p_2}(x_2)-\zeta_{n-1}(p_2;x_2)\right)\quad (n\in \N),\\
	&\Res\left(H^{(a)}_{p_1p_2,q}(x,x_1,x_2;s),-n\right)\\
	&=\frac1{(p_1+p_2)!}\lim_{s\rightarrow -n} \frac{d^{p_1+p_2}}{ds^{p_1+p_2}}\left\{(s+n)^{p_1+p_2+1}H^{(a)}_{p_1p_2,q}(x,x_1,x_2;s)\right\}\quad (n\in\N_0)\\
	&=(-1)^q\binom{q+p_1+p_2-1}{q-1}\frac{(xx_1x_2)^n}{(n-a)^{q+p_1+p_2}}\\
	&\quad+(-1)^q\sum_{k=0}^{p_{1}}\binom{k+p_2-1}{p_{2}-1}\binom{q+p_{1}-k-1}{q-1}\\
	&\qquad\qquad\times\left((-1)^k\Li_{k+p_2}(x_{2})+(-1)^{p_{2}}\zeta_n\Big(k+p_{2};x_{2}^{-1}\Big)\right)\frac{(xx_1x_2)^n}{(n-a)^{p_{1}+q-k}}\\
	&\quad+(-1)^q\sum_{k=0}^{p_{2}}\binom{k+p_{1}-1}{p_{1}-1}\binom{q+p_{2}-k-1}{q-1}\\
	&\qquad\qquad\times\left((-1)^k\Li_{k+p_1}(x_{1})+(-1)^{p_{1}}\zeta_n\Big(k+p_{1};x_{1}^{-1}\Big)\right)\frac{(xx_1x_2)^n}{(n-a)^{p_{2}+q-k}}\\
	&\quad+(-1)^q\sum_{k=0}^{p_1+p_2-1}\binom{q+p_1+p_2-k-2}{q-1}\left((-1)^k\Li_{k+1}(x)-\Li_{k+1}\Big(x^{-1}\Big)\right)\frac{(xx_1x_2)^n}{(n-a)^{q+p_1+p_2-k-1}}\\
	&\quad+(-1)^q\left(\Li_{p_1}(x_1)+(-1)^{p_1}\zeta_n\Big(p_1;x_1^{-1}\Big)\right)\left(\Li_{p_2}(x_2)+(-1)^{p_2}\zeta_n\Big(p_2;x_2^{-1}\Big)\right)\frac{(xx_1x_2)^n}{(n-a)^{q}}\\
	&\quad+(-1)^q\sum_{0\le k_1+k_2\le p_{1}-1}\binom{k_2+p_{2}-1}{p_{2}-1}\binom{q+p_{1}-k_1-k_2-2}{q-1}\nonumber\\
	&\quad\times\left((-1)^{k_1}\Li_{k_1+1}(x)-\Li_{k_1+1}\Big(x^{-1}\Big)\right)\nonumber\\
	&\quad\times
	\left((-1)^{k_2}\Li_{k_2+p_2}(x_{2})+(-1)^{p_{2}}\zeta_n\Big(k_2+p_{2};x_{2}^{-1}\Big)\right)\frac{(xx_1x_2)^n}{(n-a)^{p_{1}+q-k_1-k_2-1}}\\
	&\quad+(-1)^q\sum_{0\le k_1+k_2\le p_{2}-1}\binom{k_2+p_{1}-1}{p_{1}-1}\binom{q+p_{2}-k_1-k_2-2}{q-1}\nonumber\\
	&\quad\times\left((-1)^{k_1}\Li_{k_1+1}(x)-\Li_{k_1+1}\Big(x^{-1}\Big)\right)\nonumber\\
	&\quad\times
	\left((-1)^{k_2}\Li_{k_2+p_1}(x_{1})+(-1)^{p_{1}}\zeta_n\Big(k_2+p_{1};x_{1}^{-1}\Big)\right)\frac{(xx_1x_2)^n}{(n-a)^{p_{2}+q-k_1-k_2-1}}
\end{align*}
and
\begin{align*}
	&\Res\left(H^{(a)}_{p_1p_2,q}(x,x_1,x_2;s),-a\right)=\sum_{k_1+k_2+k_3=q-1,\atop k_1,k_2,k_3\geq 0}\left((-1)^{k_1}\Li_{{k_1}+1}(x;1-a)-x \Li_{{k_1}+1}\Big(x^{-1};a\Big) \right)\\
	&\qquad\qquad\qquad\times(-1)^{k_2+k_3}\binom{k_2+p_1-1}{p_1-1}\binom{k_3+p_2-1}{p_2-1}\Li_{k_2+p_1}(x_1;-a)\Li_{k_3+p_2}(x_2;-a)(x_1x_2)^{-1}.
\end{align*}
By Lemma \ref{lem-redisue-thm}, we have
\begin{align*}
	&\sum_{n=1}^\infty \Res\left(H^{(a)}_{p_1p_2,q}(\cdot;s),n\right)+\sum_{n=0}^\infty \Res\left(H^{(a)}_{p_1p_2,q}(\cdot;s),-n\right)\\&\quad+\sum_{n=1}^\infty\Res\left(H^{(a)}_{p_1p_2,q}(\cdot;s),-n-a\right)+\Res\left(H^{(a)}_{p_1p_2,q}(\cdot;s),-a\right)=0.
\end{align*}
Substituting the four residue results obtained above consequently proves Theorem \ref{thm-quadratic-CES-three}.
\end{proof}

\begin{exa}Setting $(p_1,p_2,q)=(1,1,2)$ in Theorem \ref{thm-quadratic-CES-three}, we have
\begin{align*}
	&(xx_1x_2)^{-1}R^{(a+1)}_{1,1;2}\Big(x_1,x_2;(xx_1x_2)^{-1}\Big)+R_{1,1;2}^{(-a)}\Big(x_1^{-1},x_2^{-1};xx_1x_2\Big)\\
	&=-\Li_{1}(x_1)\Li_{1}(x_2)\Li_2\Big((xx_1x_2)^{-1};a+1\Big)+(xx_1x_2)^{-1}\Li_{1}(x_1)R^{(a+1)}_{1;2}\Big(x_2;(xx_1x_2)^{-1}\Big)\nonumber\\
	&+(xx_1x_2)^{-1}\Li_{1}(x_2)R^{(a+1)}_{1;2}\Big(x_1;(xx_1x_2)^{-1}\Big)-3(xx_1x_2)^{-1}\Li_{4}(xx_1x_2;-a)\nonumber\\
	&-2\left((xx_1x_2)^{-1}\Li_{1}(x_{2})\Li_{3}(xx_1x_{2};-a)-R_{1;3}^{(-a)}\Big(x_2^{-1};xx_1x_2\Big)\right)\\
	&+\left((xx_1x_2)^{-1}\Li_{2}(x_{2})\Li_{2}(xx_1x_{2};-a)+R_{2;2}^{(-a)}\Big(x_2^{-1};xx_1x_2\Big)\right)\nonumber\\
	&-2\left((xx_1x_2)^{-1}\Li_{1}(x_{1})\Li_{3}(xx_1x_{2};-a)-R_{1;3}^{(-a)}\Big(x_1^{-1};xx_1x_2\Big)\right)\\
	&+\left((xx_1x_2)^{-1}\Li_{2}(x_{1})\Li_{2}(xx_1x_{2};-a)+R_{2;2}^{(-a)}\Big(x_1^{-1};xx_1x_2\Big)\right)\\
	&-2(xx_1x_2)^{-1}\left(\Li_{1}(x)-\Li_{1}\Big(x^{-1}\Big)\right)\Li_{3}(xx_1x_2;-a)\\
	&+(xx_1x_2)^{-1}\left(\Li_{2}(x)+\Li_{2}\Big(x^{-1}\Big)\right)\Li_{2}(xx_1x_2;-a)\nonumber\\
	&-(xx_1x_2)^{-1}\Li_1(x_1)\Li_1(x_2)\Li_2(xx_1x_2;-a)+\Li_1(x_1)R_{1;2}^{(-a)}\Big(x_2^{-1};xx_1x_2\Big)\nonumber\\
	&+\Li_1(x_2)R_{1;2}^{(-a)}\Big(x_1^{-1};xx_1x_2\Big)\nonumber\\
	&-\left(\Li_{1}(x)-\Li_{1}\Big(x^{-1}\Big)\right)\left((xx_1x_2)^{-1}\Li_{1}(x_{2})\Li_{2}(xx_1x_{2};-a)-R^{(-a)}_{1;2}\Big(x_{2}^{-1};xx_1x_2\Big)\right)\\
	&-\left(\Li_{1}(x)-\Li_{1}\Big(x^{-1}\Big)\right)\left((xx_1x_2)^{-1}\Li_{1}(x_{1})\Li_{2}(xx_1x_{2};-a)-R^{(-a)}_{1;2}\Big(x_{1}^{-1};xx_1x_2\Big)\right)\nonumber\\
	&+\left(\Li_{2}(x;1-a)+x \Li_{2}\Big(x^{-1};a\Big) \right)\Li_{1}(x_1;-a)\Li_{1}(x_2;-a)(x_1x_2)^{-1}\\
	&+\left(\Li_{1}(x;1-a)-x \Li_{1}\Big(x^{-1};a\Big) \right)\Li_{2}(x_1;-a)\Li_{1}(x_2;-a)(x_1x_2)^{-1}\\
	&+\left(\Li_{1}(x;1-a)-x \Li_{1}\Big(x^{-1};a\Big) \right)\Li_{1}(x_1;-a)\Li_{2}(x_2;-a)(x_1x_2)^{-1}.
\end{align*}
Setting $(p_1,p_2,q)=(2,2,2)$ in Theorem \ref{thm-quadratic-CES-three}, we have
\begin{align*}
	&(xx_1x_2)^{-1}R^{(a+1)}_{2,2;2}\Big(x_1,x_2;(xx_1x_2)^{-1}\Big)+R_{2,2;2}^{(-a)}\Big(x_1^{-1},x_2^{-1};xx_1x_2\Big)\\
	&=-\Li_{2}(x_1)\Li_{2}(x_2)\Li_2\Big((xx_1x_2)^{-1};a+1\Big)+(xx_1x_2)^{-1}\Li_{2}(x_1)R^{(a+1)}_{2;2}\Big(x_2;(xx_1x_2)^{-1}\Big)\nonumber\\
	&\quad+(xx_1x_2)^{-1}\Li_{2}(x_2)R^{(a+1)}_{2;2}\Big(x_1;(xx_1x_2)^{-1}\Big)-5(xx_1x_2)^{-1}\Li_{6}(xx_1x_2;-a)\nonumber\\
	&\quad-3\left((xx_1x_2)^{-1}\Li_{2}(x_{2})\Li_{4}(xx_1x_{2};-a)+R_{2;4}^{(-a)}\Big(x_2^{-1};xx_1x_2\Big)\right)\\
	&\quad+4\left((xx_1x_2)^{-1}\Li_{3}(x_{2})\Li_{3}(xx_1x_{2};-a)-R_{3;3}^{(-a)}\Big(x_2^{-1};xx_1x_2\Big)\right)\nonumber\\
	&\quad-3\left((xx_1x_2)^{-1}\Li_{4}(x_{2})\Li_{2}(xx_1x_{2};-a)+R_{4;2}^{(-a)}\Big(x_2^{-1};xx_1x_2\Big)\right)\nonumber\\
	&\quad-3\left((xx_1x_2)^{-1}\Li_{2}(x_{1})\Li_{4}(xx_1x_{2};-a)+R_{2;4}^{(-a)}\Big(x_1^{-1};xx_1x_2\Big)\right)\\
	&\quad+4\left((xx_1x_2)^{-1}\Li_{3}(x_{1})\Li_{3}(xx_1x_{2};-a)-R_{3;3}^{(-a)}\Big(x_1^{-1};xx_1x_2\Big)\right)\nonumber\\
	&\quad-3\left((xx_1x_2)^{-1}\Li_{4}(x_{1})\Li_{2}(xx_1x_{2};-a)+R_{4;2}^{(-a)}\Big(x_1^{-1};xx_1x_2\Big)\right)\\
	&\quad-4(xx_1x_2)^{-1}\left(\Li_{1}(x)-\Li_{1}\Big(x^{-1}\Big)\right)\Li_{5}(xx_1x_2;-a)\nonumber\\
	&\quad+3(xx_1x_2)^{-1}\left(\Li_{2}(x)+\Li_{2}\Big(x^{-1}\Big)\right)\Li_{4}(xx_1x_2;-a)\nonumber\\
	&\quad-2(xx_1x_2)^{-1}\left(\Li_{3}(x)-\Li_{3}\Big(x^{-1}\Big)\right)\Li_{3}(xx_1x_2;-a)\nonumber\\
	&\quad+(xx_1x_2)^{-1}\left(\Li_{4}(x)+\Li_{4}\Big(x^{-1}\Big)\right)\Li_{2}(xx_1x_2;-a)\nonumber\\
	&\quad-(xx_1x_2)^{-1}\Li_{2}(x_1)\Li_{2}(x_2)\Li_2(xx_1x_2;-a)-\Li_{2}(x_1)R_{2;2}^{(-a)}\Big(x_2^{-1};xx_1x_2\Big)\nonumber\\
	&\quad-\Li_{2}(x_2)R_{2;2}^{(-a)}\Big(x_1^{-1};xx_1x_2\Big)\nonumber\\
	&\quad-2\left(\Li_{1}(x)-\Li_{1}\Big(x^{-1}\Big)\right)\left((xx_1x_2)^{-1}\Li_{2}(x_{2})\Li_{3}(xx_1x_{2};-a)+R^{(-a)}_{2;3}\Big(x_{2}^{-1};xx_1x_2\Big)\right)\\
	&\quad+\left(\Li_{2}(x)+\Li_{2}\Big(x^{-1}\Big)\right)\left((xx_1x_2)^{-1}\Li_{2}(x_{2})\Li_{2}(xx_1x_{2};-a)+R^{(-a)}_{2;2}\Big(x_{2}^{-1};xx_1x_2\Big)\right)\\
	&\quad+2\left(\Li_{1}(x)-\Li_{1}\Big(x^{-1}\Big)\right)
		\left((xx_1x_2)^{-1}\Li_{3}(x_{2})\Li_{2}(xx_1x_{2};-a)-R^{(-a)}_{3;2}\Big(x_{2}^{-1};xx_1x_2\Big)\right)\\
	&\quad-2\left(\Li_{1}(x)-\Li_{1}\Big(x^{-1}\Big)\right)\left((xx_1x_2)^{-1}\Li_{2}(x_{1})\Li_{3}(xx_1x_{2};-a)+R^{(-a)}_{2;3}\Big(x_{1}^{-1};xx_1x_2\Big)\right)\\
	&\quad+\left(\Li_{2}(x)+\Li_{2}\Big(x^{-1}\Big)\right)\left((xx_1x_2)^{-1}\Li_{2}(x_{1})\Li_{2}(xx_1x_{2};-a)+R^{(-a)}_{2;2}\Big(x_{1}^{-1};xx_1x_2\Big)\right)\\
	&\quad+2\left(\Li_{1}(x)-\Li_{1}\Big(x^{-1}\Big)\right)
	\left((xx_1x_2)^{-1}\Li_{3}(x_{1})\Li_{2}(xx_1x_{2};-a)-R^{(-a)}_{3;2}\Big(x_{1}^{-1};xx_1x_2\Big)\right)\\
	&\quad+\left(\Li_{2}(x;1-a)+x \Li_{2}\Big(x^{-1};a\Big) \right)\Li_{2}(x_1;-a)\Li_{2}(x_2;-a)(x_1x_2)^{-1}\\
	&\quad+2\left(\Li_{1}(x;1-a)-x \Li_{1}\Big(x^{-1};a\Big) \right)\Li_{3}(x_1;-a)\Li_{2}(x_2;-a)(x_1x_2)^{-1}\\
	&\quad+2\left(\Li_{1}(x;1-a)-x \Li_{1}\Big(x^{-1};a\Big) \right)\Li_{2}(x_1;-a)\Li_{3}(x_2;-a)(x_1x_2)^{-1}.
\end{align*}
\end{exa}

\subsection{Generalized Hurwitz-type Cyclotomic Euler Sums}
Finally, in this subsection, we employ contour integration to present three statement theorems for the case of arbitrary order of these three types of Hurwitz-type cyclotomic Euler sums.

\begin{thm}\label{thm-parityc-CES-one} Let $x,x_1,\ldots,x_r$ be roots of unity, and $p_1,\ldots,p_r,q\geq 1$ with $(p_j,x_j) $ and $ (q,xx_1\cdots x_r)\neq (1,1)$. The
\begin{align*}
&x S_{p_1,p_2,\ldots,p_r;q}^{(a-1)}\Big(x_1,x_2,\ldots,x_r;(xx_1\cdots x_r)^{-1}\Big)\\&\quad+(-1)^{p_1+p_2+\cdots+p_r+q+r}S_{p_1,p_2,\ldots,p_r;q}^{(-a)}\Big(x_1^{-1},x_2^{-1},\ldots,x_r^{-1};xx_1\cdots x_r\Big)
\end{align*}
reduces to a combination of sums of lower orders.
\end{thm}
\begin{proof}
The proof of this theorem is based on residue calculations of the following contour integral:
\begin{align*}
&\oint\limits_{\left( \infty  \right)}F^{(a)}_{p_1p_2\cdots p_r,q}(x,x_1,x_2,\ldots,x_r;s)ds\\
&:= \oint\limits_{\left( \infty  \right)}\frac{\Phi(s;x)\phi^{(p_1-1)}(s+a;x_1)\phi^{(p_2-1)}(s+a;x_2)\cdots\phi^{(p_r-1)}(s+a;x_r)}{(p_1-1)!(p_2-1)!\cdots (p_r-1)!(s+a)^q} (-1)^{p_1+p_2+\cdots+p_r-r}ds=0.
\end{align*}
Obviously, the integrand $F^{(a)}_{p_1p_2\cdots p_r,q}(x,x_1,x_2,\ldots,x_r;s)$ possesses the following poles in the complex plane: 1. All integer points are simple poles; 2. $s=-a$ is a pole of order $p_1+p_2+\cdots+p_r+q$; 3. $s=-n-a$ (where $n$ is a positive integer) is a pole of order $p_1+p_2+\cdots+p_r$. Applying Lemma \ref{lem-redisue-thm}, we have
\begin{align}\label{residue-sums-quad}
&\sum_{n=0}^\infty \Res\left(F^{(a)}_{p_1p_2\cdots p_r,q}(\cdot;s),n\right)+\sum_{n=1}^\infty \Res\left(F^{(a)}_{p_1p_2\cdots p_r,q}(\cdot;s),-n\right)\nonumber\\&\quad+\sum_{n=1}^\infty\Res\left(F^{(a)}_{p_1p_2\cdots p_r,q}(\cdot;s),-n-a\right)+\Res\left(F^{(a)}_{p_1p_2\cdots p_r,q}(\cdot;s),-a\right)=0.
\end{align}
At integer points, which are simple zeros, the residue values can be calculated using Lemmas \ref{lem-rui-xu-two} and \ref{lem-extend-xuzhou-one} as follows:
\begin{align*}
&\Res\left(F^{(a)}_{p_1p_2\cdots p_r,q}(\cdot;s),n\right)=\frac{x^{-n}(x_1\cdots x_r)^{-n-1}}{(n+a)^q}\prod\limits_{j=1}^r\left(\Li_{p_j}(x_j;a)-\zeta_n(p_j;x_j;a-1)\right)\quad (n\in \N_0),\\
&\Res\left(F^{(a)}_{p_1p_2\cdots p_r,q}(\cdot;s),-n\right)=\frac{(xx_1\cdots x_r)^n}{(-n+a)^q}\prod\limits_{j=1}^r\left(\Li_{p_j}(x_j;a)x_j^{-1}+(-1)^{p_j}\zeta_n\Big(p_j;x_j^{-1};-a\Big)\right)\quad (n\in \N).
\end{align*}
By expanding the two residue values above and then summing them, we obtain
\begin{align*}
&\sum_{n=0}^\infty \Res\left(F^{(a)}_{p_1p_2\cdots p_r,q}(\cdot;s),n\right)+\sum_{n=1}^\infty \Res\left(F^{(a)}_{p_1p_2\cdots p_r,q}(\cdot;s),-n\right)\\
&=x S_{p_1,p_2,\ldots,p_r;q}^{(a-1)}\Big(x_1,x_2,\ldots,x_r;(xx_1\cdots x_r)^{-1}\Big)(-1)^r\\&\quad+(-1)^{p_1+p_2+\cdots+p_r+q}S_{p_1,p_2,\ldots,p_r;q}^{(-a)}\Big(x_1^{-1},x_2^{-1},\ldots,x_r^{-1};xx_1\cdots x_r\Big)\\
&\quad+\{\text{combinations of lower-order sums}\}.
\end{align*}
Applying Lemmas \ref{lem-rui-xu-one} and \ref{lem-extend-xuzhou-two}, we can also compute the latter two residue values in \eqref{residue-sums-quad}. However, the resulting sum obtained after summation will still be of order less than $r$, namely:
\begin{align*}
&\sum_{n=1}^\infty\Res\left(F^{(a)}_{p_1p_2\cdots p_r,q}(\cdot;s),-n-a\right)+\Res\left(F^{(a)}_{p_1p_2\cdots p_r,q}(\cdot;s),-a\right)\\
&\in \{\text{combinations of lower-order sums}\}.
\end{align*}
Finally, substituting these two conclusions into \eqref{residue-sums-quad} completes the proof of the theorem.
\end{proof}

\begin{thm}\label{thm-parityc-CES-two} Let $x,x_1,\ldots,x_r$ be roots of unity, and $p_1,\ldots,p_r,q\geq 1$ with $(p_j,x_j) $ and $ (q,xx_1\cdots x_r)\neq (1,1)$. The
\begin{align*}
&(x_1\cdots x_r)^{-1}\tilde{S}_{p_1,p_2,\ldots,p_r;q}^{(a-1)}\Big(x_1,x_2,\ldots,x_r;(xx_1\cdots x_r)^{-1}\Big)\\&\quad+(-1)^{p_1+p_2+\cdots+p_r+q+r}\tilde{S}_{p_1,p_2,\ldots,p_r;q}^{(-a)}\Big(x_1^{-1},x_2^{-1},\ldots,x_r^{-1};xx_1\cdots x_r\Big)
\end{align*}
reduces to a combination of sums of lower orders (It should be noted that the lower-order sum also contains Euler $R$-sums).
\end{thm}
\begin{proof}
To prove this theorem, we only need to consider contour integrals of the following form:
\begin{align*}
&\oint\limits_{\left( \infty  \right)}G^{(a)}_{p_1p_2\cdots p_r,q}(x,x_1,x_2,\ldots,x_r;s)ds\\
&:= \oint\limits_{\left( \infty  \right)}\frac{\Phi(s;x)\phi^{(p_1-1)}(s+a;x_1)\phi^{(p_2-1)}(s+a;x_2)\cdots\phi^{(p_r-1)}(s+a;x_r)}{(p_1-1)!(p_2-1)!\cdots (p_r-1)!s^q} (-1)^{p_1+p_2+\cdots+p_r-r}ds=0.
\end{align*}
The specific procedure is similar to the proof of Theorem \ref{thm-parityc-CES-one} and is therefore omitted here.
\end{proof}

\begin{thm}\label{thm-parityc-CES-three} Let $x,x_1,\ldots,x_r$ be roots of unity, and $p_1,\ldots,p_r,q\geq 1$ with $(p_j,x_j) $ and $ (q,xx_1\cdots x_r)\neq (1,1)$. The
\begin{align*}
&(xx_1\cdots x_r)^{-1}R_{p_1,\ldots,p_r;q}^{(a+1)}\Big(x_1,\ldots,x_r;(xx_1\cdots x_r)^{-1}\Big)\\
&\quad+(-1)^{p_1+\cdots+p_r+q+r}R_{p_1,\ldots,p_r;q}^{(-a)}\Big(x_1^{-1},\ldots,x_r^{-1};xx_1\cdots x_r\Big)
\end{align*}
reduces to a combination of sums of lower orders.
\end{thm}
\begin{proof}
To prove this theorem, we only need to consider contour integrals of the following form:
\begin{align*}
&\oint\limits_{\left( \infty  \right)}H^{(a)}_{p_1p_2\cdots p_r,q}(x,x_1,x_2,\ldots,x_r;s)ds\\
&:= \oint\limits_{\left( \infty  \right)}\frac{\Phi(s;x)\phi^{(p_1-1)}(s;x_1)\phi^{(p_2-1)}(s;x_2)\cdots\phi^{(p_r-1)}(s;x_r)}{(p_1-1)!(p_2-1)!\cdots (p_r-1)!(s+a)^q} (-1)^{p_1+p_2+\cdots+p_r-r}ds=0.
\end{align*}
The proof of this theorem is omitted as it follows a similar line of reasoning to that of Theorem \ref{thm-parityc-CES-one}.
\end{proof}

\section{Parity Results of Multiple Hurwitz Polylogarithm Function}
In this section, by employing the relationship between Hurwitz-type cyclotomic Euler $S$-sums and the multiple Hurwitz  polylogarithm function, combined with the results from the preceding sections, we present the parity results and explicit formulas for the multiple Hurwitz  polylogarithm function of depth at most 3.

According to definition of linear Hurwitz-type cyclotomic Euler $S$-sums and double Hurwitz polylogarithm function with $2$-variables, we have
\begin{align*}
	S_{p;q}^{(a)}(x;y)=\Li_{p,q}(x,y;a+1)+\Li_{p+q}(xy;a+1)
\end{align*}
Therefore, we can derive the following corollary regarding the parity of multiple Hurwitz polylogarithm function.
\begin{cor}\label{thm-linearCES-cor}Let $x,y$ be $N$-th roots of unity and $a\in \mathbb{C}\setminus \Z$, and $p,q\ge1$ with $(p,y),(q,y)\ne(1,1)$. Then
\begin{align}
\Li_{p,q}(x,y;a)-(-1)^{p+q}xy\Li_{p,q}\Big(x^{-1},y^{-1};1-a\Big)
\end{align}
reduces to a combination of single Hurwitz polylogarithm functions.
\end{cor}
\begin{exa}Let $(p,q)=(2,2)$ and $(p,q)=(2,3)$ in Corollary \ref{thm-linearCES-cor}, we have
	\begin{align*}
		&\Li_{2,2}\Big(x,y;a\Big)-xy\Li_{2,2}\Big(x^{-1},y^{-1};1-a\Big)\\
		&=-\Li_{4}(xy;a)+xy\Li_{4}\Big((xy)^{-1};1-a\Big)+\Li_{2}(x;a)\Li_2(y;a)+y\Li_2(x;a)\Li_2\Big(y^{-1};1-a\Big)\\
		&\quad-xy\Li_{4}\Big((xy)^{-1};1-a\Big)+2xy\left(\Li_{1}\Big((xy)^{-1};1-a\Big)-(xy)^{-1}\Li_{1}(xy;a)\right)\Li_{3}\Big(y^{-1}\Big)\\
		&\quad-xy\left(\Li_{2}\Big((xy)^{-1};1-a\Big)+(xy)^{-1}\Li_{2}(xy;a)\right)\Li_{2}\Big(y^{-1}\Big)\\
		&\quad+2xy\left((xy)^{-1}\Li_{1}(xy;a)-\Li_{1}\Big((xy)^{-1};1-a\Big)\right)\Li_{3}(x)\\
		&\quad-xy\left((xy)^{-1}\Li_{2}(xy;a)+\Li_{2}\Big((xy)^{-1};1-a\Big)\right)\Li_{2}(x),\\
		     &\Li_{2,3}\Big(x,y;a\Big)+xy\Li_{2,3}\Big(x^{-1},y^{-1};1-a\Big)\\
		&=-\Li_{5}(xy;a)-xy\Li_{5}\Big((xy)^{-1};1-a\Big)+\Li_{2}(x;a)\Li_3(y;a)-y\Li_2(x;a)\Li_3\Big(y^{-1};1-a\Big)\\
		&\quad+xy\Li_{5}\Big((xy)^{-1};1-a\Big)-3xy\left(\Li_{1}\Big((xy)^{-1};1-a\Big)-(xy)^{-1}\Li_{1}\Big(xy;a\Big)\right)\Li_{4}\Big(y^{-1}\Big)\\
		&\quad+xy\left(\Li_{2}\Big((xy)^{-1};1-a\Big)+(xy)^{-1}\Li_{2}(xy;a)\right)\Li_{3}\Big(y^{-1}\Big)\\
		&\quad-3xy\left((xy)^{-1}\Li_{1}\Big(xy;a\Big)-\Li_{1}\Big((xy)^{-1};1-a\Big)\right)\Li_{4}(x)\\
		&\quad+2xy\left((xy)^{-1}\Li_{2}\Big(xy;a\Big)+\Li_{2}\Big((xy)^{-1};1-a\Big)\right)\Li_{3}(x)\\
		&\quad-xy\left((xy)^{-1}\Li_{3}\Big(xy;a\Big)-\Li_{3}\Big((xy)^{-1};1-a\Big)\right)\Li_{2}(x).
	\end{align*}
\end{exa}
According to definition of Hurwitz-type cyclotomic quadratic Euler $S$-sums and multiple Hurwitz polylogarithm function with $3$-variables, we have
\begin{align*}
	&S^{(a)}_{p_1,p_2;q}(x_1,x_2;x)\\
	&=\sum_{n=1}^\infty\frac{\zeta_{n}(p_1;x_1;a)\zeta_{n}(p_2;x_2;a)}{(n+a)^q}x^n\\
	&=\sum_{n=1}^\infty\frac{(\zeta_{n}(p_1;x_1;a)-\Li_{p_1}(x_1;a+1))\zeta_{n}(p_2;x_2;a)}{(n+a)^q}x^n+\Li_{p_1}(x_1;a+1)\sum_{n=1}^\infty\frac{\zeta_{n}(p_2;x_2;a)}{(n+a)^q}x^n\\
	&=\sum_{n=1}^\infty\frac{(\zeta_{n}(p_1;x_1;a)-\Li_{p_1}(x_1;a+1))\zeta_{n-1}(p_2;x_2;a)}{(n+a)^q}x^n+\sum_{n=1}^\infty\frac{\zeta_{n}(p_1;x_1;a)-\Li_{p_1}(x_1;a+1)}{(n+a)^{p_2+q}}(xx_2)^n\\
	&\quad+\Li_{p_1}(x_1;a+1)\sum_{n=1}^\infty\frac{\zeta_{n-1}(p_2;x_2;a)+\frac{x_2^{n}}{(n+a)^{p_2}}}{(n+a)^q}x^n\\
	&=-\Li_{p_2,q,p_1}(x_2,x,x_1;a+1)-\Li_{p_2+q,p_1}(x_2x,x_1;a+1)\\
	&\quad+\Li_{p_1}(x_1;a+1)\left(\Li_{p_2,q}(x_2,x;a+1)+\Li_{p_2+q}(xx_2;a+1)\right).
\end{align*}
Therefore, we can derive the following corollary regarding the parity of multiple Hurwitz polylogarithm function with $3$-variables with a direct calculation.
\begin{cor}\label{thm-quadratic-CES-one-cor}Let $x,y,z$ be $N$th-roots of unity and $a\in \mathbb{C}\setminus \Z$, and $p,q,r \ge1$ with $(p,x),(q,y)$ and $(r,z)\ne(1,1)$. Then
	\begin{align*}
		\Li_{p,q,r}(x,y,z;a)	+(-1)^{p+q+r}xyz\Li_{p,q,q}\Big(x^{-1},y^{-1},z^{-1};1-a\Big)
	\end{align*}
reduces to a combination of multiple Hurwitz polylogarithm function with depth $\leq 2$.
\end{cor}
\begin{exa}
	Let $(p,q,r)=(2,1,2)$ in Corollary \ref{thm-quadratic-CES-one-cor}, we have
\begin{align*}
	&\Li_{2,1,2}(x,y,z;a)	-xyz\Li_{2,1,2}\Big(x^{-1},y^{-1},z^{-1};1-a\Big)\\
	&=-\Li_{3,2}(xy,z;a)+\Li_{2}(z;a)\left(\Li_{2,1}(x,y;a)+\Li_{3}(xy;a)\right)+xyz\Li_{3,2}\Big((xy)^{-1},z^{-1};1-a\Big)\\
	&\quad-xyz\Li_{2}\Big(z^{-1};1-a\Big)\left(\Li_{2,1}\Big(x^{-1},y^{-1};1-a\Big)+\Li_{3}\Big((xy)^{-1};1-a\Big)\right)\\
	&\quad-S_{2;3}^{(a-1)}(z;xy)-S_{2;3}^{(a-1)}(x;yz)-\Li_{2}(z;a)S_{2;1}^{(a-1)}(x;y)-\Li_{2}(x;a)S_{2;1}^{(a-1)}(z;y)\nonumber\\
	&\quad-xy\Li_{2}(z;a)S_{2;1}^{(-a)}\Big(x^{-1};y^{-1}\Big)-yz\Li_{2}(x;a)S_{2;1}^{(-a)}\Big(z^{-1};y^{-1}\Big)\nonumber\\
	&\quad+xyz\Li_{5}\Big((xyz)^{-1};1-a\Big)+\Li_{2}(z;a)\Li_{3}(xy;a)+\Li_{2}(x;a)\Li_{3}(yz;a)\nonumber\\
	&\quad+\Li_{2}(z;a)\Li_{2}(x;a)\Li_1(y;a)-y\Li_{2}(z;a)\Li_{2}(x;a)\Li_1\Big(y^{-1};1-a\Big)\nonumber\\
	&\quad-xyz\left(\Li_{1}\Big((xyz)^{-1};1-a\Big)-(xyz)^{-1}\Li_{1}(xyz;a)\right)\Li_{4}\Big(y^{-1}\Big)\\
	&\quad+xyz\left(\Li_{2}\Big((xyz)^{-1};1-a\Big)+(xyz)^{-1}\Li_{2}(xyz;a)\right)\Li_{3}\Big(y^{-1}\Big)\nonumber\\
	&\quad-xyz\left(\Li_{3}\Big((xyz)^{-1};1-a\Big)-(xyz)^{-1}\Li_{3}(xyz;a)\right)\Li_{2}\Big(y^{-1}\Big)\\
	&\quad+xyz\left(\Li_{4}\Big((xyz)^{-1};1-a\Big)+(xyz)^{-1}\Li_{4}(xyz;a)\right)\Li_{1}\Big(y^{-1}\Big)\nonumber\\
	&\quad-xyz\left(\Li_{1}\Big((xyz)^{-1};1-a\Big)-(xyz)^{-1} \Li_{1}(xyz;a)\right)\left(\Li_{2}(z)\Li_{2}\Big(y^{-1}\Big)+S_{2;2}\Big(z^{-1};y^{-1}\Big)\right)\nonumber\\
	&\quad+2xyz\left(\Li_{1}\Big((xyz)^{-1};1-a\Big)-(xyz)^{-1} \Li_{1}(xyz;a)\right)\left(\Li_{3}(z)\Li_{1}\Big(y^{-1}\Big)-S_{3;1}\Big(z^{-1};y^{-1}\Big)\right)\nonumber\\
	&\quad+xyz\left(\Li_{2}\Big((xyz)^{-1};1-a\Big)+(xyz)^{-1} \Li_{2}(xyz;a)\right)\left(\Li_{2}(z)\Li_{1}\Big(y^{-1}\Big)+S_{2;1}\Big(z^{-1};y^{-1}\Big)\right)\nonumber\\
	&\quad-xyz\left(\Li_{1}\Big((xyz)^{-1};1-a\Big)-(xyz)^{-1} \Li_{1}(xyz;a)\right)\left(\Li_{2}(x)\Li_{2}\Big(y^{-1}\Big)+S_{2;2}\Big(x^{-1};y^{-1}\Big)\right)\nonumber\\
	&\quad+2xyz\left(\Li_{1}\Big((xyz)^{-1};1-a\Big)-(xyz)^{-1} \Li_{1}(xyz;a)\right)\left(\Li_{3}(x)\Li_{1}\Big(y^{-1}\Big)-S_{3;1}\Big(x^{-1};y^{-1}\Big)\right)\nonumber\\
	&\quad+xyz\left(\Li_{2}\Big((xyz)^{-1};1-a\Big)+(xyz)^{-1} \Li_{2}(xyz;a)\right)\left(\Li_{2}(x)\Li_{1}\Big(y^{-1}\Big)+S_{2;1}\Big(x^{-1};y^{-1}\Big)\right)\\
	&\quad+xyz\Li_{2}(x)\left(\Li_{3}\Big((xyz)^{-1};1-a\Big)-(xyz)^{-1} \Li_{3}\Big(xyz;a\Big)\right)\\
	&\quad+2xyz\Li_{3}(x)\left(\Li_{2}\Big((xyz)^{-1};1-a\Big)+(xyz)^{-1}\Li_{2}\Big(xyz;a\Big)\right)\nonumber\\
	&\quad+3xyz\Li_{4}(x)\left(\Li_{1}\Big((xyz)^{-1};1-a\Big)-(xyz)^{-1}\Li_{1}\Big(xyz;a\Big)\right)\\
    &\quad+xyz\Li_{2}(z)\left(\Li_{3}\Big((xyz)^{-1};1-a\Big)-(xyz)^{-1} \Li_{3}\Big(xyz;a\Big)\right)\\
    &\quad+2xyz\Li_{3}(z)\left(\Li_{2}\Big((xyz)^{-1};1-a\Big)+(xyz)^{-1}\Li_{2}\Big(xyz;a\Big)\right)\nonumber\\
    &\quad+3xyz\Li_{4}(z)\left(\Li_{1}\Big((xyz)^{-1};1-a\Big)-(xyz)^{-1}\Li_{1}\Big(xyz;a\Big)\right)\\
	&\quad+xyz\Li_{2}(z)\Li_2(x)\left(\Li_{1}\Big((xyz)^{-1};1-a\Big)-(xyz)^{-1}\Li_{1}\Big(xyz;a\Big)\right).
\end{align*}
\end{exa}

All results in references \cite{Wang-Xu2025,XC2025} can be obtained from the main theorems of this paper by setting $a=\pm 1/2$.

Finally, based on our observations and computations, we conclude this paper by proposing the following conjecture concerning the parity of multiple Hurwitz polylogarithm functions of arbitrary depth.
\begin{con}\label{conjparitycmHZv} Let $r>1$ and $x_1,\ldots,x_r$ be roots of unity and $a\in \mathbb{C}\setminus \Z$, and $k_1,\ldots,k_r\geq 1$ with $(k_r,x_r)\neq (1,1)$. Then
\begin{align*}
\Li_{k_1,\ldots,k_r}(x_1,\ldots,x_r;a)+(-1)^{k_1+\cdots+k_r}x_1x_2\cdots x_r\Li_{k_1,\ldots,k_r}\Big(x^{-1}_1,\ldots,x^{-1}_r;1-a\Big)
\end{align*}
can be expressed in terms of a rational linear combination of multiple Hurwitz polylogarithm functions with depth $< r$.
\end{con}
Evidently, Conjecture 1.2 regarding cyclotomic multiple $t$-values in Reference \cite{Wang-Xu2025} represents a special case of this conjecture.
Actually, the parity result conjectured in Conjecture \ref{conjparitycmHZv} also has a counterpart, known as the symmetry conjecture, which can be stated as follows:
\begin{con}\label{conjsymetrycmHZv}
Let $x_1,\ldots,x_r$ be roots of unity and $a\in \mathbb{C}\setminus \Z$. For $\bfk=(k_1,\ldots,k_r)\in \N^r$ and $(k_1,x_1),\ (k_r,x_r)\neq (1,1)$, we have
\begin{align*}
\Li_{k_1,\ldots,k_r}(x_1,\ldots,x_r;a)\equiv  (-1)^{k_1+\cdots+k_r-1} (x_1\cdots x_r) \Li_{k_r,\ldots,k_1}(x^{-1}_r,\ldots,x_1^{-1};1-a)\quad (\text{mod products}),
\end{align*}
where ``mod products" means discarding all product terms of cyclotomic multiple Hurwitz zeta values with depth $<r$.
\end{con}

Using the antipode relations for multiple polylogarithms established in \cite{XuZhao2020d}, Conjecture \ref{conjsymetrycmHZv} can be directly derived from Conjecture \ref{conjparitycmHZv}.
The symmetric result presented in Conjecture \ref{conjsymetrycmHZv} should be regularizable. In a recent paper by Charlton and Hoffman \cite{CharltonHoffman-MathZ2025}, they established symmetric results for regularized multiple $t$-values. The methods in their paper may also be applied to prove the symmetric result for cyclotomic multiple Hurwitz zeta values in Conjecture \ref{conjsymetrycmHZv} of this paper, and even extend to the regularized setting. This, in turn, could help prove the parity result in Conjecture \ref{conjparitycmHZv}  and its regularized version as well. Due to space limitations, we will not elaborate in detail here. The authors plan to pursue this direction in subsequent work. Interested readers are also encouraged to explore this topic further.

\medskip

{\bf Declaration of competing interest.}
The author declares that he has no known competing financial interests or personal relationships that could have
appeared to influence the work reported in this paper.

{\bf Data availability.}
No data was used for the research described in the article.

{\bf Acknowledgments.}
Hongyuan Rui thanks Ms. Lin for her help.

\end{document}